\documentclass[11pt]{article}

\usepackage{graphics}
\usepackage{color}
\usepackage{cancel}
\usepackage{tikz}
\usetikzlibrary{matrix,arrows}
\usepackage{amsmath}
\usepackage{amsthm}
\usepackage{amssymb}
\usepackage{verbatim}
\usepackage{setspace}

\usepackage[left=3cm,top=3cm,right=3cm,bottom=3cm]{geometry}

\newcommand{\op}{\mathrm{op}}
\renewcommand{\hat}{\widehat}

\newcommand{\pri}{\mathrm{PrIdl}}
\newcommand{\prf}{\mathrm{PrFil}}

\theoremstyle{plain}
\newtheorem{thm}{Theorem}

\newtheorem{lem}[thm]{Lemma}
\newtheorem{cor}[thm]{Corollary}
\newtheorem{prop}[thm]{Proposition}

\theoremstyle{definition}
\newtheorem{dfn}[thm]{Definition}
\newtheorem{exa}[thm]{Example}

\newtheorem{rem}[thm]{Remark}

\numberwithin{thm}{section}

\let\oldmarginpar\marginpar
\renewcommand\marginpar[1]{\-\oldmarginpar[\raggedleft\footnotesize #1]{\raggedright\footnotesize #1}}

\newcounter{mnotes}
\newcommand{\dom}{\mathrm{dom}}
\newcommand{\upset}{\mathord{\uparrow}}
\newcommand{\downset}{\mathord{\downarrow}}

\newcommand{\m}{\ensuremath{\ominus}}

\title{Priestley duality for MV-algebras and beyond}

\author{Wesley Fussner, Mai Gehrke, Sam van Gool, and Vincenzo Marra}

\begin{document}
\maketitle

\begin{abstract}
We provide a new perspective on extended Priestley duality for a large class of distributive lattices equipped with binary double quasioperators. Under this approach, non-lattice binary operations are each presented as a pair of partial binary operations on dual spaces. In this enriched environment, equational conditions on the algebraic side of the duality may more often be rendered as first-order conditions on dual spaces. In particular, we specialize our general results to the variety of MV-algebras, obtaining a duality for these in which the equations axiomatizing MV-algebras are dualized as first-order conditions.
\end{abstract}

\section{Introduction}

MV-algebras have been the subject of intense study principally for two reasons: First, because they provide the equivalent algebraic semantics for \L{}ukasiewicz many-valued propositional logic \cite{BP1989}, and second, because of their deep connection to lattice-ordered abelian groups via the Mundici functor \cite{Mun1986}. Despite sharp interest in a duality-theoretic analysis from both of these directions, MV-algebras have been notoriously resistant to study from the perspective of Priestley duality. Although extensions of Priestley duality (see, e.g., \cite{Gol1989}) provide the necessary tools to dualize expansions of bounded distributive lattices by additional operations (such as the monoid operation of MV-algebras), the axioms defining MV-algebras are not easily dualized. Indeed, the characteristic identity
\begin{align}\label{mvlaw}\tag{MV6}
 \neg(\neg x \oplus y)\oplus y = \neg(\neg y \oplus x)\oplus x
\end{align}
of MV-algebras is not canonical \cite{GePr02}. This creates a significant obstacle in rendering (MV6) in terms of an equivalent first-order condition on extended dual spaces, substantially hindering a transparent characterization of appropriate duals (see, e.g., \cite{CC2006}).

One approach, which offers at least a theoretical advantage, is extended Priestley duality for so-called \emph{double quasioperator algebras} \cite{GePr07a,GePr07b}. The latter comprise a huge class of lattice-ordered algebraic structures, including MV-algebras and, more generally, semilinear residuated binars (see, e.g., \cite{FJ2019}). For these, first-order dual conditions are guaranteed under the condition that we double the non-lattice operations of arity two or higher. Dually, this requires, a priori, two relations per additional operation for the duality of \cite{Gol1989}, or one derived one, as shown in \cite{GePr07b}. From the work in \cite{GePr07a}, it is clear that an alternative approach for double quasioperator algebras is to witness them dually by \emph{two partial operations} each. This is what we will do here.

Although much of our inquiry applies more generally, we focus on a class of double quasioperator algebras that contains, \emph{inter alia}, MV-algebras when presented in the signature including the bounded lattice operations and $\m$ (the co-residual of the monoid operation $\oplus$). We call such algebraic structures $\m$-algebras, and provide a duality for them by enriching the Priestley duals of their bounded distributive lattice reducts by a pair of partial binary operations, which together dualize the operation $\m$. This pair of partial binary operations reflects the two natural extensions of the operation $\m$ to the canonical extension of its corresponding bounded distributive lattice, and thus our work is anchored throughout in the theory of canonical extensions.

The introduction of both of the personalities of $\m$ on dual spaces affords a more expressive environment. As an application illustrating this expressive power, we specialize our duality for $\m$-algebras to MV-algebras. In this setting, the law (MV6) may be rendered as a transparent first-order condition on the two partial binary operations (see Section \ref{subsec:mv6}). We expect the ideas laid out in this paper to support many similar applications, and note that other applications of this duality theory are already under way (see, e.g., \cite{GJR2019}).

The paper is structured as follows. Section \ref{sec:preliminaries} presents necessary background information regarding canonical extensions, residuated algebraic structures, and Priestley duality. Following these preliminary remarks, Section \ref{sec:Priestley duality} develops our Priestley duality for $\m$-algebras. Section \ref{sec:mv} then specializes the duality for $\ominus$-algebras to MV-algebras, and specifically provides an analysis of the role of the defining condition (MV6). We finish our discussion in Section \ref{sec:examples}, which offers some illustrative examples of the theory developed in earlier sections.

\section{Preliminaries}\label{sec:preliminaries}

Duality theory for distributive lattices descends from a key insight of Birkhoff \cite{Bir1937}: Each finite distributive lattice is determined up to isomorphism by its poset of join-irreducible elements. Putting aside certain exceptional cases, this fails badly for infinite distributive lattices; in fact, infinite distributive lattices may have no join-irreducible elements at all. To salvage Birkhoff's insight in such cases, we must introduce enough idealized join-irreducible elements in order to recover the original lattice. In the present treatment, we use in parallel two intertwined approaches to accomplishing this: An explicitly algebraic perspective arising from canonical extensions of bounded distributive lattices, and an order-topological perspective anchored in Priestley duality. Here we summarize pertinent facts about these two approaches and their relationship. For general background, see \cite{GeJo2004}. Note also that \cite{GGM2014} recalls pertinent facts regarding the canonical extensions of MV-algebras. More background regarding MV-algebras themselves may be found in Section \ref{sec:mv}.

\subsection{Canonical extensions}
In a complete lattice $C$, $a\in C$ is said to be \emph{compact} if for any $S\subseteq C$ with  $a\leq\bigvee S$ there exists some finite $S'\subseteq S$ with $a\leq\bigvee S'$. A lattice $C$ is said to be \emph{algebraic} if $C$ is complete and for any $a\in C$ there exists a set $S$ of compact elements of $C$ with $a=\bigvee S$. A lattice $C$ is said to be \emph{dually algebraic} if its opposite lattice $C^\mathsf{op}$ is algebraic, and \emph{doubly algebraic} if it is both algebraic and dually algebraic. Given a complete distributive lattice $C$, an element $a\in C$ is said to be \emph{completely join-irreducible} if $a\in S$ for every $S\subseteq C$ with $a=\bigvee S$, and dually $a$ is said to be \emph{completely meet-irreducible} if $a\in S$ for every $S\subseteq C$ with $a=\bigwedge S$. We denote the collections of completely join-irreducible and completely meet-irreducible elements of $C$ by $J^\infty(C)$ and $M^\infty(C)$, respectively. For a doubly algebraic distributive lattice $C$, there are mutually-inverse poset isomorphisms $\kappa\colon J^\infty(C)\to M^\infty(C)$ and \mbox{$\kappa^{-1}\colon M^\infty(C)\to J^\infty(C)$} given, for $x\in J^\infty(C)$ and $y\in M^\infty(C)$, by
$$\kappa (x) = \bigvee \{a\in A : x\nleq a\}$$
and
$$\kappa^{-1}(y) = \bigwedge \{b\in A : b\nleq y\},$$
respectively. Generalizing the finite case, every doubly algebraic distributive lattice $C$ is isomorphic to the lattice of down-sets of the poset $J^\infty(C)$ and, equivalently, of the poset $M^\infty(C)$.

If $A$ is a sublattice of a doubly algebraic lattice $C$, we may consider the closure of $A$ in $C$ by arbitrary (not just finite) meets and joins. We say that $a\in C$ is a \emph{filter element} if $a$ is a (possibly infinitary) meet of elements of $A$ and an \emph{ideal element} if $a$ is a (possibly infinitary) join of elements of $A$. The sets of filter elements of $C$ and ideal elements of $C$ are denoted by $F(C)$ and $I(C)$, respectively. Of course, these notions depend on the choice of the sublattice $A$ as well as the lattice $C$, but for our purposes the choice of $A$ will be obvious from context and we do not explicitly refer to it in our notation. Note that filter and ideal elements were respectively called \emph{closed} and \emph{open} elements in \cite{GeJo2004}. Following this topological analogy, we define the \emph{interior map} $\mathrm{int}\colon C\to I(C)$ by
$$\mathrm{int}(a) = \bigvee\{x\in I(C) \mid x\leq a\}.$$

If $A$ is a sublattice of a doubly algebraic lattice $C$, then we say that $A$ is \emph{separating} in $C$ if $J^\infty(C)\subseteq F(C)$. Note that this is equivalent to $M^\infty(C)\subseteq I(C)$, and also to the demand that for all $x,y\in C$ with $x\not\leq y$ there exists $a\in A$ such that $x\not\leq a$ and $y\leq a$. In the same setting, we say that $A$ is \emph{compact} in $C$ if whenever $S,T\subseteq A$ with $\bigwedge S\leq \bigvee T$ in $C$, there exist finite $S'\subseteq S$ and $T'\subseteq T$ such that $\bigwedge S'\leq \bigvee T'$.

A \emph{canonical extension} of a bounded distributive lattice $A$ is a doubly algebraic lattice $A^\delta$ that contains $A$ as a separating, compact (bounded) sublattice. Every bounded distributive lattice has a canonical extension, and this is unique up to an isomorphism fixing $A$. In light of this, we refer to \emph{the} canonical extension of a bounded distributive lattice $A$, and denote it by $A^\delta$.

The canonical extension of an arbitrary bounded distributive lattice $A$ is a completion of $A$, and therefore provides an embedding of $A$ into the doubly algebraic lattice $A^\delta$. Moreover, it may be shown that $A$ is \emph{dense} in $A^\delta$ in the sense that each element of $A^\delta$ is at once a join of meets of elements of $A$ and a meet of joins of elements of $A$. If $A$ and $B$ are bounded distributive lattices and $f\colon A\to B$ is any map, density provides two obvious candidates for extending $f$ to a map $A^\delta\to B^\delta$, namely the \emph{$\sigma$-} and \emph{$\pi$-extensions} given respectively by
$$f^\sigma(x) = \bigvee\left\{\bigwedge f([p,u]\cap A) \mid p\in F(A^\delta), u\in I(A^\delta),\text{ and }p\leq x\leq u \right\}$$
and
$$f^\pi(x) = \bigwedge\left\{\bigvee f([p,u]\cap A) \mid p\in F(A^\delta), u\in I(A^\delta),\text{ and } p\leq x\leq u \right\},$$
where $[p,u] = \{y\in A^\delta \mid p\leq y\leq u\}$. Note that in general $f^\sigma$ and $f^\pi$ do not coincide, but $f^\sigma\leq f^\pi$ always holds. For this reason, the $\sigma$- and $\pi$-extensions are also called the \emph{lower} and \emph{upper} extensions. Note that $f^\sigma(x) = f^\pi(x)$ whenever $x$ is a filter or ideal element. In the event that $f^\sigma(x)=f^\pi(x)$ for all $x\in A^\delta$, $f$ is called \emph{smooth}. When $f$ is smooth, we will write the common value of $f^\sigma$ and $f^\pi$ as $f^\delta$. A basic fact regarding canonical extensions is that if $f$ preserves join or meet, or reverses at least one of the two, then $f$ is smooth.

We note that for any lattices $A$ and $B$, we have $(A^\delta)^\mathsf{op}\cong (A^\mathsf{op})^\delta$ and $(A\times B)^\delta\cong A^\delta\times B^\delta$. This fact permits us to extend the definitions of the $\sigma$- and $\pi$-extensions (given above for unary maps) to those of arbitrary finite arity in the obvious way. In particular, the definitions of the $\sigma$- and $\pi$-extensions may be extended to arbitrary finitary algebraic operations expanding bounded distributive lattices. We freely apply this observation in the case of $\ominus$-algebras in the work to follow.

In the sequel, we will occasionally have need of the following canonicity result (drawn from \cite[Thm. 4.6]{GeJo1994} as well as its dual statement).

\begin{prop}\label{prop:canonicity}
Let $A$ be a distributive lattice and $f_i\colon A\to A$, $i\in I$.
\begin{enumerate}
\item If each $f_i$ preserves finite joins, then every identity that holds in $(A, \{f_i\}_{i\in I})$ also holds in $(A^\delta, \{f^\sigma_i\}_{i\in I})$.
\item If each $f_i$ preserves finite meets, then every identity that holds in $(A, \{f_i\}_{i\in I})$ also holds in $(A^\delta, \{f^\pi_i\}_{i\in I})$.
\end{enumerate}
\end{prop}

\subsection{Residuation, co-residuation, and double quasioperators}
In order to express the theory of adjoints and (co-)residuated structures in sufficient generality, we must introduce some technical notation. This material is necessarily burdened by some heavy bookkeeping, so in order to provide intuition we explain its conceptual origins after we have introduced the relevant definitions and results.

If $A$ and $B$ are posets and $f\colon A\to B$, then a map \mbox{$f^\sharp\colon B\to A$} is called an \emph{upper adjoint} of $f$ if for all $x\in A$, $y\in B$,
$$f(x)\leq y \iff x\leq f^\sharp(y).$$
Dually, a map $f^\flat\colon B\to A$ is called a \emph{lower adjoint} of $f$ if for all $x\in B$, $y\in A$,
$$x\leq f(y) \iff f^\flat(x)\leq y.$$
If $f\colon A\times B\to C$ is a binary map, then for each $a\in A$, $b\in B$, we define left and right translation maps $L_{f,a}\colon B\to C$ and $R_{f,b}\colon A\to C$ by $L_{f,a}(x)=f(a,x)$ and $R_{f,b}(x)=f(x,b)$, denoting these respectively by $L_a$ and $R_b$ if $f$ is understood from context. In the event that $L_{f,a}$ has an upper adjoint (respectively, lower adjoint) for each $a\in A$, by a slight abuse of notation we define a map $L_f^\sharp\colon A\times C\to B$ by $L_f^\sharp(x,y) = L_{f,x}^\sharp(y)$ (and, respectively, $L_f^\flat\colon A\times C\to B$ by $L_f^\flat(x,y) = L_{f,x}^\flat(y)$). We call the map $L_f^\sharp$ (respectively, $L_f^\flat$) the \emph{left (co-)residual} of $f$. The \emph{right (co-)residual} $R_f^\sharp\colon C\times B\to A$ (respectively $R_f^\flat\colon C\times B\to A$) is analogously defined by $R_f^\sharp(x,y)=R_{f,y}^\sharp(x)$ (respectively $R_f^\flat(x,y)=R_{f,y}^\flat(x)$), provided that each of the maps $R_{f,y}^\sharp$ (respectively, $R_{f,y}^\flat$) exist. When $f$ is clear from context, we denote the left and right (co-)residuals by $L^\sharp$ and $R^\sharp$ (respectively, $L^\flat$ and $R^\flat$). If $f$ is such that both $L^\sharp$ and $R^\sharp$ (respectively, $L^\flat$ and $R^\flat$) exist, then we say that $f$ is \emph{(co-)residuated}.

Note that if $A$ and $B$ are complete lattices and $f\colon A\to B$ preserves arbitrary joins (meets), then $f$ has an upper adjoint (respectively, lower adjoint). For bounded distributive lattices $A$ and $B$, if $f\colon A\to B$ preserves finitary meets (or, respectively, joins), $f^\delta$ preserves arbitrary meets (respectively, joins). Similar comments apply when $f$ converts finitary meets to joins (or vice versa). This entails that if $f$ preserves binary joins (meets), then $f^\delta$ has an upper adjoint (lower adjoint). In the event that $f\colon A\times B\to C$ is a binary map that preserves arbitrary joins (meets) in each coordinate, the extension $f^\sigma$ (respectively $f^\pi$) is \mbox{(co-)residuated}. Indeed, if $f$ is \mbox{(co-)residuated} to begin with, then the left and right \mbox{(co-)residuals} of $f^\sigma$ (respectively $f^\pi$) are exactly $(L^\sharp)^\pi$ and $(R^\sharp)^\pi$ (respectively $(L^\flat)^\pi$ and $(R^\flat)^\pi$). Because adjoints are ubiquitous in canonical extensions, we will typically drop parentheses appearing due to successive applications of $\sigma$, $\pi$, $\delta$, $\sharp$, and $\flat$ in order to ease our notational burden (writing, e.g., $(f^\delta)^\sharp$ as $f^{\delta\sharp}$).

Given a bounded distributive lattice $A$, we denote by $A^1$ the lattice $A$ itself and by $A^\mathsf{op}$ its opposite lattice. By an \emph{order type}\footnote{In \cite{GePr07a} and \cite{GePr07b}, the symbol $\partial$ was used for what is denoted $\mathsf{op}$ in the present paper.} we mean an element of $\{1,\mathsf{op}\}^n$ for some positive integer $n$, and by a \emph{binary order type} we mean an order type for which $n=2$. For a binary order type $\varepsilon=(\varepsilon_1,\varepsilon_2)$, a \emph{double quasioperator of type $\varepsilon$} is a map $f\colon A^{\varepsilon_1}\times A^{\varepsilon_2}\to A$ that preserves both meet and join in each coordinate (for more information, see \cite{GePr07a} and \cite{GePr07b}, which provide a study of canonical extensions of double quasioperators of arbitrary, not-necessarily-binary order type. See also \cite{FJ2019} for an algebraic study of equational conditions defining residuated double quasioperators.)

Although it applies much more generally, the content of this study is motivated by the following situation. Suppose that $A$ is a bounded distributive lattice, that $\oplus$ is a co-residuated double quasioperator of type $(1,1)$ on $A$, and that the right co-residual $\m$ of $\oplus$ is a double quasioperator of type $(1,\mathsf{op})$.\footnote{Note that if $f\colon A\times A\to B$ is a (co-)residuated map satisfying $f(x,y)=f(y,x)$ for all $x,y\in A$, one may readily show that $L^\sharp(x,y)=R^\sharp(y,x)$ (respectively, $L^\flat(x,y)=R^\flat(y,x)$). Because the left and right (co-)residuals coincide in this setting, it is common practice to take only one of them as primitive.} Then $\oplus$ is the right residual of $\m$. Both $\oplus$ and $\m$ are typically written in infix notation, with e.g. the co-residuation property being expressed by the stipulation that
$$x\leq y\oplus z \iff x\m z\leq y$$
for all $x,y,z\in A$. In this setting, $\m^\sigma$ has a right residual on $A^\delta$ given by $\oplus^\pi$ (see, e.g., \cite[Prop.~5.3]{GGM2014}). Said differently, for any $x,y\in A^\delta$,
$$x\oplus^\pi y = R_{\m^\sigma}^\sharp (x,y).$$
Accordingly, the map $R_{\m^\sigma}^\sharp$ is a manifestation of the (extension of the) operation $\oplus$ on $A^\delta$. In the setting where $\oplus$ is absent from $A$ to begin with (as in the $\m$-algebras considered in the sequel), an analogue of its extension is nevertheless present in the guise of $R_{\m^\sigma}^\sharp$. Due to the importance of the latter operation, we abbreviate it by the more evocative $\m^{\sigma\sharp}$. We also adopt infix notation, writing
$$x\m^{\sigma\sharp}y = R_{\m^\sigma}^\sharp(x,y).$$

\subsection{Priestley duality and its connection to the canonical extension}

The following well-known result is often called Esakia's Lemma in the more restrictive setting of a modal operator (see, e.g., \cite{Ge2014} for a treatment in the language of canonical extensions). To express this lemma concisely, we call a subset $S$ of a poset \emph{up-directed} if every pair of elements of $S$ have a common upper bound in $S$, and \emph{down-directed} if every pair of elements of $S$ have a common lower bound in $S$. A poset $X$ is called a \emph{directed-complete partial order} (or \emph{dcpo}) if every up-directed subset of $X$ has a supremum. Likewise, $X$ is called a \emph{dual dcpo} if every down-directed subset of $X$ has an infimum.
\begin{lem}\label{lem:Esakia}
Let $A$ and $B$ be bounded distributive lattices and let $f\colon A\to B$ be order-preserving.
\begin{enumerate}
\item If $S$ is an up-directed subset of $I(A^\delta)$, then $f^\delta\left (\bigvee S\right ) = \bigvee \{f^\delta(x) \mid x\in S\}.$
\item If $S$ is a down-directed subset of $F(A^\delta)$, then $f^\delta\left (\bigwedge S\right ) = \bigwedge\{f^\delta(x) \mid x\in S\}.$
\end{enumerate}
\end{lem}

If $A$ is a distributive lattice, then for each $a\in A$ we may define
$$\hat{a} = \{x\in M^\infty(A^\delta) \mid a\not\leq x\}.$$
The collection $\{\hat{a} \mid a\in A\}\cup\{(\hat{a})^\mathsf{c} \mid a\in A\}$ gives a subbase for a topology $\tau$ on $M^\infty(A^\delta)$, and equipped with this topology $X=M^\infty(A^\delta)$ forms a compact ordered topological space that is \emph{totally order-disconnected}, i.e., if $x,y\in X$ with $x\not\leq y$ then there exists a clopen down-set $U\subseteq X$ such that $y\in U$ and $x\notin U$. Compact, totally order-disconnected ordered topological spaces are called \emph{Priestley spaces}. When endowed with continuous isotone maps for morphisms, Priestley spaces form a category that is dually equivalent to the category of bounded distributive lattices and bounded lattices homomorphisms \cite{Pr70,Pr72}. As depicted here, one functor of this equivalence associates to a bounded distributive lattice $A$ the Priestley space $(M^\infty(A^\delta),\leq,\tau)$, and associates to a homomorphism $f\colon A\to B$ of bounded distributive lattices the continuous isotone map $(f^\delta)^\sharp\restriction M^\infty(B^\delta)$.  It is noteworthy that the underlying order of a Priestley space is both a dcpo and a dual dcpo.

If $A$ is a bounded distributive lattice, a \emph{Priestley dual space} of $A$ is any Priestley space $X$ such that $A$ is isomorphic to the lattice of clopen down-sets of $X$. All Priestley dual spaces of $A$ are isomorphic. As usual when discussing particular representatives from an isomorphism class, we sometimes refer to a given Priestley dual space of $A$ as ``the'' Priestley dual space of $A$. However, there are many (isomorphic, but not identical) choices for this. For example, $(M^\infty(A^\delta),\leq,\tau)$ is a Priestley dual space of $A$. Another example may be found by taking the set $\pri(A)$ of prime ideals of $A$ endowed with the inclusion order $\subseteq$ and the topology $\sigma$ generated by all sets of the form
$$\{x\in\pri(A) \mid a\in x\}, \{x\in\pri(A) \mid a\notin x\},$$
where $a\in A$. The structure $(\pri(A), \subseteq, \sigma)$ is then a Priestley dual space of $A$. The map sending a prime ideal $I$ of $A$ to the completely meet-irreducible element $\bigvee I$ of $A^\delta$ is an isomorphism between $(\pri(A), \subseteq, \sigma)$ and $(M^\infty(A^\delta), \leq, \tau)$.

Of course, we could just as well work with $J^\infty(A^\delta)$ instead of $M^\infty(A^\delta)$ (by exploiting the isomorphism $\kappa$) or with the collection of prime filters $\prf(A)$ of $A$ (by exploiting the fact that $\prf(A)\cong\pri(A)^\mathsf{op}$ via $x\mapsto A\setminus x$). Suitably modifying the subbase for the topology in the obvious ways, we obtain Priestley dual spaces of $A$ in each instance.

In the presence of double quasioperators, one is naturally led to toggle between $J^\infty(A^\delta)$ and $M^\infty(A^\delta)$ (and between prime filters and prime ideals). To facilitate this, for any Priestley dual space $X$ of $A$ we define four Priestley space isomorphisms
$${I_{(-)}^X}\colon X\to \pri(A),$$
$${F_{(-)}^X}\colon X\to \prf(A),$$
$$\mu^X\colon X\to M^\infty(A^\delta),$$
$$\nu^X\colon X\to J^\infty(A^\delta),$$
obtained by composing the isomorphisms outlined above. Since the domain $X$ will always be clear from context, for ease of notation we will abbreviate $I_{(-)}^X$, $F_{(-)}^X$, $\mu^X$, and $\nu^X$ by $I_{(-)}$, $F_{(-)}$, $\mu$, and $\nu$, respectively. The isomorphisms are connected via
$$I_x = A\cap\downset\mu(x),$$
$$\mu(x) = \bigvee I_x,$$
$$F_x = A\cap\upset\nu(x),$$
$$\nu(x) = \bigwedge F_x,$$
$$\kappa(\nu(x)) = \mu(x),$$
$$F_x^\mathsf{c} = I_x.$$
In the sequel, we will often employ the above to switch between different Priestley dual spaces of a given distributive lattice $A$.\footnote{Readers who are more comfortable working with a particular, fixed representation of ``the'' Priestley dual space of $A$ may read this paper with that choice in mind. For example, one could, if one wished, define ``the'' Priestley dual space of $A$ to be $(\pri(A), \subseteq, \sigma)$ and read every mention of the dual of $A$ according to that definition. All results in the sequel remain true without modification. However, we emphasize that a ``neutral'' perspective that toggles between particular representations of the Priestley dual of $A$ is the most natural environment for this paper's content.}

Our treatment above began from the perspective of canonical extensions, but one may also start from the point of view of Priestley duality. From this perspective, we note that the canonical extension of $A$ may be realized as the collection of down-sets of the Priestley dual $X$ via  $a\mapsto\{x\in X \mid x\in \hat{a}\}$.

\section{Priestley duality for \m-algebras}\label{sec:Priestley duality}
We introduce \m-algebras and develop an extended Priestley duality for these based on Priestley spaces with additional partial operations.

\subsection{Definitions of \m-algebras and their extended dual spaces}

\begin{dfn}\label{dfn:malgebras}
A \m-algebra, $(A,\m)$, is a bounded distributive lattice $A$ equipped with a binary operation \m\ satisfying:
\begin{enumerate}
\item \m\ is a double quasioperator of type $(1,\mathsf{op})$. That is, for all $a,b,c\in A$,
\begin{align*}
(a\wedge b)\m c=(a\m c)\wedge(b\m c) &\ \text{ and }\ (a\vee b)\m c=(a\m c)\vee(b\m c)\\
 a\m(b\wedge c)=(a\m b)\vee(a\m c) &\ \text{ and }\ a\m(b\vee c)=(a\m b)\wedge(a\m c)
\end{align*}
\item \m\ is normal as an operator. That is, for all $a\in A$ we have $0\m a=0$ and $a\m 1=0$.
\item For all $a \in A$, $a \m 0 = a$.
\end{enumerate}
\end{dfn}

\begin{rem}
   We focus on double quasioperators of type $(1,\mathsf{op})$ because we will apply the duality developed in this section to the operator \m\ of an MV-algebra. We also impose the second and third conditions because these hold in the \m-reduct of MV-algebras and because it will somewhat simplify the partiality of the operations dual to \m. However, a generalization of the results proved in this section for \m-algebras may be obtained for general double quasioperator algebras.
\end{rem}

In seeking a functional dual of \m\ it will be more convenient to have the dual of the following derived operation available. Throughout the following, we denote the dual space of $A$ by $X$.

\begin{dfn}
Given a \m-algebra $(A,\m)$, the associated \emph{negation} is the operation
\[
\neg\colon A \to A, \ a\mapsto 1\m a.
\]
\end{dfn}

From the properties of \m\ it follows that $\neg \colon A^\op \to A$ is a lattice homomorphism, i.e., $\neg$ reverses both binary meet and join, sends $0$ to $1$ and $1$ to $0$.
Accordingly, the dual of $\neg$ is a continuous order-reversing map on $X$.

\begin{dfn}
We define the dual $i$ of $\neg$ to be the continuous order-reversing map on $X$ dual to the homomorphism $\neg\colon A^\op \to A$.
Since $\neg$ is unary and reverses binary meets (and joins), we have $\neg^\sigma=\neg^\pi$, and thus we denote the unique extension by $\neg^\delta \colon (A^\delta)^\op \to A^\delta$. We have
\[
\forall u,v\in A^\delta\ \left(\ \neg^\delta u\ \leq\ v\ \iff\ \neg^{\delta\sharp} v\ \leq\ u\ \right)
\]
and
\[
\forall u,v\in A^\delta\ \left(\ u\ \leq\ \neg^\delta v\ \iff\ v\leq \neg^{\delta\flat} u\ \right).
\]
In particular, for each $x\in X$ we have
\[
\nu(i(x)) = \neg^{\delta\sharp}\mu(x) \text{ and } \mu(i(x)) = \neg^{\delta\flat}\nu(x),
\]
or, said differently,
\[ F_{i(x)} = \{a \in A \ \mid\ \neg a \in I_x\} \text{ and } I_{i(x)} = \{a \in A \ \mid\ \neg a \in F_x\}.\]
\end{dfn}

In order to obtain a dual description of \m, we consider the upper and lower canonical extensions, $\m^\pi$ and $\m^\sigma$, which are in general different.

Consider first $\m^\pi$. From the general theory of canonical extensions, we know that $\m^\pi$ preserves arbitrary non-empty meets in the first coordinate and reverses arbitrary non-empty joins in the second coordinate. In addition, $\m^\pi$ preserves finite joins in the first coordinate and reverses finite meets in the second coordinate because \m\ has these properties and they are canonical by Proposition \ref{prop:canonicity}. Thus, for each $y\in X$, the map $R_{\m^\pi, \nu(y)}\colon A^\delta \to A^\delta$ given by $R_{\m^\pi,\nu(y)}(u) = u\m^\pi\nu(y)$ preserves arbitrary non-empty meets. Viewing this as a map $R_{\m^\pi,\nu(y)}\colon A^\delta \to [0,\neg^\delta\nu(y)]$ (which is possible because the map $R_{\m^\pi,\nu(y)}$ has $\neg^\delta\nu(y)$ as the maximum of its image), it is completely meet preserving. The corresponding map $A^\delta\times X\to [0,\neg^\delta\nu(y)]$ defined by $(u,y)\mapsto R_{\m^\pi,\nu(y)}(u)$ thus has a right co-residual $\m^{\pi\flat}\colon [0, \neg^\delta \nu(y)]\times X \to A^\delta$, which is uniquely determined by the property:
\[
\forall u,v\in A^\delta\ \left(\ u\ \leq\ v\m^\pi \nu(y)\ \iff\ (\ u\leq\neg^\delta\nu(y)\ \text{ and }\  u\m^{\pi\flat}\nu(y)\leq v\ )\right).
\]
Also, for each $y\in X$, since $R_{\m^\sigma,\mu(y)}\colon u\mapsto \ u\m^\sigma\mu(y)$ preserves arbitrary joins (including the empty join), the corresponding binary map $(u,y)\mapsto R_{\m^\sigma,\mu(y)}(u)$ has right residual \mbox{$\m^{\sigma\sharp}\colon A^\delta\times X\to A^\delta$}, uniquely determined by:
\[
\forall u,v\in A^\delta\ \left(\ u\m^\sigma \mu(y)\, \leq\, v\ \iff\ u\,\leq\, v\m^{\sigma\sharp}\mu(y)\right).
\]
The following fundamental fact about double quasioperators is what makes it possible to obtain functional duals for these operations.

\begin{prop}
Let $(A,\m)$ be a \m-algebra, $X$ the dual space of $A$, and let $y\in X$. Then the following hold:
\begin{enumerate}
\item If $j\in J^\infty(A^\delta)$ and $j\leq\neg^\delta \nu(y)$, then either $j \m^{\pi\flat}\nu(y)=0$ or $j \m^{\pi\flat}\nu(y)\in J^\infty(A^\delta)$;
\item If $m\in M^\infty(A^\delta)$, then either $m \m^{\sigma\sharp}\mu(y)=1$ or $m \m^{\sigma\sharp}\mu(y)\in M^\infty(A^\delta)$.
\end{enumerate}
\end{prop}
\begin{proof}
See the proof of \cite[Thm. 4.4]{GePr07a}, or also \cite[Prop. 6.4]{GGM2014}.
\end{proof}
Now it is interesting to understand for which completely join irreducibles $j\in [0, \neg^\delta \nu(y)]$ we have $j \m^{\pi\flat}\nu(y)=0$, and for which completely meet irreducibles we have $m \m^{\sigma\sharp}\mu(y)=1$.

\begin{prop}
Let $(A,\m)$ be a \m-algebra, $X$ the dual space of $A$, and let $x,y\in X$. The following three conditions are equivalent:
\begin{enumerate}
\item[(i)] $\nu(x)\leq\neg^\delta\nu(y)$;
\item[(ii)] $i(x)\nleq y$;
\item[(iii)] $(x,y)\in\bigcup_{a\in A}\left[\widehat{\neg a}\times\widehat{a}\right]$.
\end{enumerate}
Moreover, if these conditions hold then $\nu(x)\m^{\pi\flat}\nu(y)\neq0$. Furthermore, the following three conditions are equivalent:
\begin{enumerate}
\item[(iv)] $\mu(x)\m^{\sigma\sharp}\mu(y)\neq1$;
\item[(v)] $y\leq i(x)$;
\item[(vi)] $(x,y)\in\bigcap_{a\in A}\left([\widehat{\neg a}\times X]\cup[X\times\widehat{a}]\right)$.
\end{enumerate}
\end{prop}

\begin{proof}
Note that
\begin{align*}
\nu(x)\leq\neg^\delta\nu(y) &\iff \neg^\delta\nu(y)\nleq \mu(x)\\
                                           &\iff \neg^{\delta\sharp}\mu(x)\nleq \nu(y)\\
                                           &\iff \nu(i(x))\nleq\nu(y)\\
                                           &\iff i(x)\nleq y.
\end{align*}
For the last condition of the first equivalence we start from
\[
\nu(x)\leq\neg^\delta\nu(y)=\bigvee\{\neg a\mid y\in\widehat{a}\}.
\]
Since $\nu(x)$ is completely join-irreducible, this is equivalent to the existence of $a\in A$ with $y\in\widehat{a}$ and $x\in\widehat{\neg a}$, thus yielding the last condition. For the moreover statement, note that $\nu(x)\not\leq 0 = 0\m^\pi\nu(y)$, so $\nu(x)\m^{\pi\flat}\nu(y)\neq 0$ by residuation.

For the second set of equivalences, we have
\begin{align*}
\mu(x)\m^{\sigma\sharp}\mu(y)\neq1 &\iff 1\nleq\mu(x)\m^{\sigma\sharp}\mu(y)\\
                                           &\iff \neg^\delta\mu(y)\nleq\mu(x)\\
                                           &\iff \neg^{\delta\sharp} \mu(x) \nleq \mu(y) \\
                                           &\iff \nu(y)\leq \nu(i(x))\\
                                           &\iff y \leq i(x).
\end{align*}
For the last condition of the second equivalence we start from
\[
\nu(x)\leq\neg^\delta\mu(y)=\bigwedge\{\neg a\mid y\not\in\widehat{a}\}.
\]
By definition of infima, this is equivalent to  $y\in\widehat{a}$ or $x\in\widehat{\neg a}$ for all $a\in A$ thus yielding the last condition.
\end{proof}

We are now ready to define the two partial functions which will account for \m\ on the dual of $A$.

\begin{dfn}\label{dfn:plusstar}
Let $(A,\m)$ be a \m-algebra and $X$ the dual space of $A$. Let
\[
\dom(+) := \bigcap_{a\in A}\left([\,\widehat{\neg a}\times X]\cup[X\times\widehat{a}\,]\right) = \{(x,y) \in X^2 \mid y \leq i(x)\}
\]
and let
\[
+\ \colon \dom(+) \to X, \ x+y=\mu^{-1}(\mu(x)\m^{\sigma\sharp}\mu(y)).
\]
Further define
\[
\dom(\star) := \bigcup_{a\in A}\left[\,\widehat{\neg a}\times\widehat{a}\,\right] = \{(x,y) \in X^2 \mid i(x) \nleq y\}
\]
and
\[
\star\ \colon \dom(\star) \to X, \ x\star y=\nu^{-1}( \nu(x)\m^{\pi\flat}\nu(y)).\\
\]
We call the tuple $(X,i,+,\star)$ the \emph{extended dual space} of $(A,\m)$.
\end{dfn}

The definitions of $+$ and $\star$ can be rephrased in terms of prime filters and prime ideals, as follows.
\begin{lem}\label{lem:plusstardef}
Let $(A,\m)$ be a \m-algebra and $(X,i,+,\star)$ its extended dual space.

For any $(x,y) \in \dom(+)$, we have:
\begin{enumerate}
\item $F_{x + y} = \{a \in A \mid \text{for all } b \in A, \text{ if } b \in I_y \text{ then } a \ominus b \in F_x) \}$,
\item $I_{x + y} = \{a \in A \mid \text{there exists } b \in A \text{ such that } b \in I_y \text{ and } a \ominus b \in I_x\}.$
\end{enumerate}
For any $(x,y) \in \dom(\star)$, we have:
\begin{enumerate}

\item[3.] $F_{x \star y} = \{a \in A \mid \text{there exists } b \in A \text{ such that } b \in F_y \text{ and } a \ominus b \in F_x\}.$
\item[4.] $I_{x \star y} = \{a \in A \mid \text{for all } b \in A, \text{ if } b \in F_y \text{ then } a \ominus b \in I_x\},$
\end{enumerate}
\end{lem}

\begin{proof}
Let $(x,y) \in \dom(+)$. For any $a \in A$, we have
\begin{align*}
\nu(x + y) \leq a &\iff a \nleq \mu(x+y) = \mu(x) \m^{\sigma\sharp} \mu(y) \\
&\iff a \m^\sigma \mu(y) \nleq \mu(x) \\
&\iff \nu(x) \leq a \m^\sigma \mu(y) = \bigwedge \{ a \m b \mid b \leq \mu(y)\}.
\end{align*}
From this, the first item follows by the definition of infimum. The second item now follows, because $I_{x+y}$ is the complement of $F_{x+y}$. The proof of the last two items is similar.
\end{proof}

\subsection{Properties of extended dual spaces of \m-algebras}
For the following, it is convenient to introduce the following notation. Suppose that $X$ is a Priestley space whose topology is $\tau$. By $\tau^\uparrow$ we mean the topology consisting of all open up-sets in $\tau$, and by $\tau^\downarrow$ we mean the topology consisting of all open down-sets in $\tau$.
See, for example, \cite[Sec.~2]{GGM2014} for more information on the relationships between $\tau$, $\tau^\uparrow$, and $\tau^\downarrow$.

\begin{prop}\label{prop:dualproperties}
The extended dual space $(X,i,+,\star)$ of a \m-algebra $(A,\m)$ has the following properties.
\begin{enumerate}
\item The domain of $+$ is a closed down-set.
\item The domain of $\star$ is an open down-set.
\item The function $+$ is (jointly) upper continuous, i.e., $+$ is continuous when viewed as a function from $\dom(+)$ to $(X,\tau^\uparrow)$, where $\dom(+)$ is equipped with the subspace topology of the product $(X,\tau^\uparrow) \times (X,\tau^\uparrow)$.
\item The function $\star$ is (jointly) lower continuous, i.e., $\star$ is continuous when viewed as a function from $\dom(\star)$ to $(X,\tau^\downarrow)$, where $\dom(\star)$ is equipped with the subspace topology of the product $(X,\tau^\downarrow) \times (X,\tau^\downarrow)$.
\item The function $+$ is order-preserving, i.e., for any $(x,y) \in \dom(+)$, $x' \leq x$ and $y' \leq y$, we have $x' + y' \leq x + y$.
\item The function $\star$ is order-preserving, i.e., for any $(x,y) \in \dom(\star)$, $x \leq x'$ and $y \leq y'$, we have $x \star y \leq x' \star y'$.
\end{enumerate}

\end{prop}
\begin{proof}
(1) holds because $\dom(+)$ is an intersection of clopen down-sets, and (2) holds because $\dom(\star)$ is a union of clopen down-sets. For (3), by definition of the topology $\tau^\uparrow$ on $X$, it suffices to prove that $(+)^{-1}(\hat{a})$ is closed in $\dom(+)$ for every $a \in A$. Indeed, using Lemma~\ref{lem:plusstardef}, we have
\[ (+)^{-1}(\hat{a}) = \dom(+) \cap \bigcap_{b \in A} \left[ (\hat{a \ominus b} \times X) \cup (X \times \hat{b})\right], \]
which is a closed set in the subspace $\dom(+)$ of $(X,\tau^\uparrow) \times (X,\tau^\uparrow)$. Similarly, for (4), it suffices to prove that $(\star)^{-1}(\hat{a})$ is open in $\dom(\star)$ for every $a \in A$. Again using Lemma~\ref{lem:plusstardef}, we have
\[ (\star)^{-1}(\hat{a}) = \dom(\star) \cap \bigcup_{b \in A} \left[ \hat{a \ominus b} \cap \hat{b} \right], \]
which is an open set in the subspace $\dom(\star)$ of $(X,\tau^\downarrow) \times (X,\tau^\downarrow)$. Items (5) and (6) follow because both $\m^{\sigma\sharp}$ and $\m^{\pi\flat}$ are order preserving in each coordinate.
\end{proof}

Inspection of the foregoing proof attests that,
\emph{mutatis mutandis},
the previous proposition applies
 more generally
to any bounded distributive lattice expanded by a double quasioperator. The next proposition,
however,
 uses the
 two
additional
 defining
properties of $\m$-algebras.

\begin{prop}\label{prop:dualproperties2}
The extended dual space $(X,i,+,\star)$ of a \m-algebra $(A,\m)$ has the following properties.
\begin{enumerate}
\item The function $+$ is expanding, i.e., for any $(x,y) \in \dom(+)$, $x \leq x + y$.
\item For any $x \in X$, there exists $y_x \in X$ such that $(x,y_x) \in \dom(+)$ and $x + y_x = x$.
\end{enumerate}
\end{prop}

\begin{proof}
For (1), note first that $u \ominus^\sigma 0 = u$ for any $u \in A^\delta$, by Proposition \ref{prop:canonicity}. In particular, for any $x, y\in X$, we have
\[\mu(x) \leq \mu(x)\ominus^{\sigma\sharp} 0 \leq \mu(x)\m^{\sigma\sharp}\mu(y),\]
where, in the second inequality, we use that $\m^{\sigma\sharp}$ is order preserving. Thus, $x \leq x + y$ if $(x,y) \in \dom(+)$.
For (2) note that, for any $x \in X$, we have that $\nu(x) \leq \nu(x) \m^\sigma 0$. 
Hence, by residuation, we have that
\[\nu(x) \nleq \mu(x) \m^{\sigma\sharp} 0 = \bigwedge_{y \in X} \mu(x) \m^{\sigma\sharp} \mu(y),\]
where the second equality holds because $\m^{\sigma\sharp}$ is a right residual, hence preserves arbitrary meets in the second coordinate.
By definition of infimum, pick $y_x \in X$ such that $\nu(x) \nleq \mu(x) \m^{\sigma\sharp} \mu(y_x)$. In particular, $\mu(x) \m^{\sigma\sharp} \mu(y_x) \neq 1$, so $(x,y_x) \in \dom(+)$. By residuation and the definition of $+$, we have $x + y_x \leq x$. By (1), we must have $x + y_x = x$.
\end{proof}

The following lemma is an immediate consequence of \cite[Prop. 4.2]{GePr07b}.
For the benefit of the reader not familiar with the notation of that paper, we record the proof in our setting.
\begin{lem}\label{lem:Rdop}
Let $(A,\m)$ be a \m-algebra with extended dual space $(X,i,+,\star)$.
For any $x, w_1, w_2 \in X$ with $(x,w_1) \in \dom(+)$ and $(x,w_2) \in \dom(+)$, there exists $w_0 \in X$ such that $(x,w_0) \in \dom(+)$, $w_1 \leq w_0$, $w_2 \leq w_0$, and either $x + w_0 = x + w_1$ or $x + w_0 = x + w_2$.
\end{lem}

\begin{proof}
Write $z := \nu(x+w_1) \vee \nu(x+w_2)$. Note that, for $k \in \{1,2\}$, we have
\[ \nu(x) \leq \nu(x+w_k) \m^\sigma \mu(w_k) \leq z \m^\sigma \mu(w_k),\]
where the first inequality follows from the definition of $+$ and residuation, and the second inequality from order preservation of $\m^\sigma$ in the first coordinate.
Therefore, since $\m^\sigma$ sends $\vee$ to $\wedge$ in the second coordinate, we deduce that
\[ \nu(x) \leq z \m^\sigma (\mu(w_1) \vee \mu(w_2)).\]
Now, because $\m^\sigma$ reverses arbitrary meets in the second coordinate, we obtain
\[ \nu(x) \leq z \m^\sigma \bigwedge \{ \mu(w_0) \ | \ \mu(w_1) \vee \mu(w_2) \leq \mu(w_0) \} = \bigvee\{ z \m^\sigma \mu(w_0) \ | \ \mu(w_1) \vee \mu(w_2) \leq \mu(w_0)\}.\]
As $\nu(x)$ is completely join irreducible, pick $w_0 \in X$ with $\mu(w_1) \vee \mu(w_2) \leq \mu(w_0)$ and $\nu(x) \leq z \m^\sigma \mu(w_0)$. Clearly, $w_1 \leq w_0$ and $w_2 \leq w_0$. It also follows from $\nu(x)\leq z\m^\sigma\mu(w)$ that $z \nleq \mu(x) \m^{\sigma\sharp} \mu(w_0)$, so that $(x,w_0) \in \dom(+)$. It remains to prove that $x + w_0 \in \{x+w_1,x+w_2\}$.

Since $\m^\sigma$ preserves joins in the first coordinate, we have
\[ \nu(x) \leq z \m^\sigma \mu(w_0) = (\nu(x+w_1) \m^\sigma \mu(w_0)) \vee (\nu(x+w_2) \m^\sigma \mu(w_0)).\]
As $\nu(x)$ is join irreducible, pick $k \in \{1,2\}$ such that $\nu(x) \leq \nu(x+w_k) \m^\sigma \mu(w_0)$. We claim that $x + w_0 = x + w_k$. Indeed, it follows from the definition of $+$ and residuation that $\mu(x+w_0) \leq \mu(x+w_k)$, so $x+w_0 \leq x+w_k$. Moreover, since $+$ is order-preserving and $w_k \leq w_0$, we also have $x + w_k \leq x + w_0$. Thus, $x + w_0 = x + w_k$, as required.
\end{proof}

In Proposition~\ref{prop:translation}, we exhibit important properties of the unary left translation operations $L_{+,x}$ induced by the binary operation $+$.
\begin{prop}\label{prop:translation}
Let $X$ be the extended dual space of a $\m$-algebra $A$. For any $x \in X$, the function $L_x=L_{+,x}$ is well-defined as a map ${\downarrow} i(x) \to {\uparrow} x$, and has a totally ordered image and an upper adjoint $L_x^\sharp \colon {\uparrow} x \to {\downarrow} i(x)$.\end{prop}
\begin{proof}
Note that $L_x$ is indeed well-defined considered as a map with domain ${\downarrow} i(x)$ and codomain ${\uparrow} x$ by Definition~\ref{dfn:plusstar} and Proposition~\ref{prop:dualproperties2}.7. To see that the image of $L_x$ is totally ordered, let $w_1, w_2 \in {\downarrow} i(x)$. By Lemma~\ref{lem:Rdop}, pick $w_0 \in {\downarrow} i(x)$ such that $w_1, w_2 \leq w_0$ and, say, $x + w_0 = x + w_1$. Since $w_2 \leq w_0$ and $+$ is order-preserving by Proposition~\ref{prop:dualproperties}.5, we have $x + w_2 \leq x + w_0 \leq x + w_1$. In case instead $x + w_0 = x + w_2$, we obtain in the same way $x + w_1 \leq x + w_2$.

Finally, we show the existence of the upper adjoint. Let $z \in {\uparrow} x$ be arbitrary, and consider the set
\[ S := \{ y \in {\downarrow} i(x) \ | \ x + y \leq z \}.\]
Note that $S$ has a supremum: By Lemma~\ref{lem:Rdop}, $S$ is an up-directed set, by Proposition~\ref{prop:dualproperties}.8, $S$ is non-empty, and $\leq$ is a directed complete partial order. Define $k(x,z) := \sup S$. We claim that $k(x,-)$ is the desired upper adjoint $L_x^\sharp$.

Clearly, $k(x,z) \leq i(x)$; moreover, $L_x(y)\leq z$ trivially implies $y\leq k(x,z)$ for any $y\leq i(x)$. For the converse, it suffices to show that $x + k(x,z) \leq z$. 
Note first that, for any $y \in {\downarrow}i(x)$, we have $x+  y \leq z$ if, and only if, $\nu(x) \leq \nu(z) \m^\sigma \mu(y)$, using the definitions of $+$, $\mu$, $\nu$ and residuation. Therefore,
\begin{align*}
\nu(x) &\leq \bigwedge \{\nu(z) \m^\sigma \mu(y) \ | \ y \in S \}  &\text{(definition of $S$)} \\
&= \nu(z) \m^\sigma \bigvee \{\mu(y) \ | \ y \in S\} &\text{(Lemma \ref{lem:Esakia})} \\
&= \nu(z) \m^\sigma \mu(k(x,z)) &\text{($\mu$ is order-isomorphism)}.
\end{align*}
Hence, $x + k(x,z) \leq z$, as required.
\end{proof}
\begin{rem}
A unary version of the binary function $k$ defined in the proof of Proposition~\ref{prop:translation} has appeared before in the literature on duality for MV-algebras; also see Remark 7.7 in \cite{GGM2014}.
\end{rem}

Since the operations $+$ and $\star$ come from one and the same operation on $A$, it is natural that they should be related. In Proposition~\ref{prop:starminplus} we make this precise.

\begin{prop}\label{prop:starminplus}
Let $(A,\m)$ be a \m-algebra, $X$ the extended dual space of $A$, and $(x,y)\in \dom(\star)$. Then
\[
x\star y=\inf \{x+w\mid (x,w)\in\dom(+) \text{ and }w\nleq y\},
\]
where the infimum is with respect to the order on $X$.
\end{prop}

\begin{proof}
We first prove that
\begin{equation} \label{eq:starplusrelation}
\mu(x \star y) = \mathrm{int} ( \mu(x) \m^{\sigma\sharp} \nu(y) ).
\end{equation}
To this end, note that, for any $a \in A$, we have
\begin{align*}
a \leq \mu(x \star y) &\iff \nu(x \star y) \nleq a & \\
                      &\iff \nu(x) \nleq a \m^\pi \nu(y) = a \m^\sigma \nu(y) & (*) \\
                      &\iff a \m^\sigma \nu(y) \leq \mu(x) & \\
                      &\iff a \leq \mu(x) \m^{\sigma\sharp} \nu(y), &
\end{align*}
where the equality marked $(*)$ follows from the fact that $\m^\sigma$ and $\m^\pi$ agree on ideal elements of $A^\delta \times (A^\delta)^\op$.
Thus, the equality (\ref{eq:starplusrelation}) follows since the completely meet-irreducible $\mu(x\star y)$ is an ideal element.

Now, since $\m^{\sigma\sharp}$ is the right residual of a $(1,\mathsf{op})$ double quasioperator, it preserves arbitrary meets in both coordinates, so
\[ \mu(x) \m^{\sigma\sharp} \nu(y) = \bigwedge \{ \mu(x) \m^{\sigma\sharp} \mu(w) \ | \ \nu(y) \leq \mu(w)\}.\]
 Moreover, since $\mu(x) \m^{\sigma\sharp} \mu(w) = \mu(x+w)$ when $(x,w) \in \dom(+)$, and $\mu(x) \m^{\sigma\sharp} \mu(w) = 1$ otherwise, we obtain
\begin{align}\label{eq:secondrelation}
\mu(x) \m^{\sigma\sharp} \nu(y) &= \bigwedge \{\mu(x + w) \mid (x,w) \in \dom(+) \text{ and } w \nleq y\}.
\end{align}
Now, for any $z \in X$, $\mu(z)$ is an ideal element. Therefore, again using the fact that $\m^\sigma$ and $\m^\pi$ agree on ideal elements, for any $z \in X$,
\begin{align*}
\mu(z) \leq \mu(x \star y) &\iff \nu(x \star y) \nleq \mu(z) & \\
                      &\iff \nu(x) \nleq \mu(z) \m^\pi \nu(y) = \mu(z) \m^\sigma \nu(y) &\\
&\iff \mu(z) \leq \mu(x + w) \text{ for all } (x,w) \in \dom(+) \text{ such that } w \nleq y &(\text{using (\ref{eq:secondrelation})})
\end{align*}
Since $\mu$ is an order isomorphism, it follows from this equivalence that $x \star y$ is indeed the greatest lower bound of the set on the right hand side.
\end{proof}

\begin{rem}[Comparison with \cite{GePr07b}]
\label{rem:dualproperties}
We compare the properties of the binary operations $+$ and $\star$ that we proved in the above to the properties of the ternary relations $R$ and $S$ introduced in a more general setting in \cite{GePr07b}. The properties (1)--(4) in our Prop.~\ref{prop:dualproperties} correspond to ($R_{\mathrm{top}}^\epsilon$) and ($S_{\mathrm{top}}^\epsilon$) in \cite{GePr07b}, and (5) and (6) correspond to ($R_{\mathrm{ord}}^\epsilon$) and ($S_{\mathrm{ord}}^\epsilon$) in \cite{GePr07b}. The conjunction of properties (7) and (8) in Prop.~\ref{prop:dualproperties} corresponds to the specific axiom $a \m 0 = a$ of \m-algebras, which is not considered in \cite{GePr07b}. The property proved in Lem.~\ref{lem:Rdop} corresponds to the property ($R_\mathrm{dop}^\epsilon$) in \cite{GePr07b}, and the property proved in Prop.~\ref{prop:starminplus} corresponds to ($RS^\epsilon_{\leq}$) and ($RS^\epsilon_{\geq}$) in \cite{GePr07b}. 
\end{rem}

\subsection{\m-spaces and their dual algebras}
The above properties can be used to completely characterize the extended dual spaces of \m-algebras, which we call \emph{\m-spaces}. Indeed, in Section \ref{sec:morphisms} we will define an appropriate notion of morphisms with which \m-spaces become a category, and we will prove a duality theorem between $\m$-algebras and $\m$-spaces (Theorem~\ref{thm:mduality}).

The following provides the needed characterization of the extended dual spaces of \m-algebras.
\begin{dfn}\label{def:mspace}
A \emph{\m-space} is a tuple $(X,i,+,\star)$ where:
 \begin{enumerate}
 \item $X$ is a Priestley space,
 \item $i \colon X \to X$ is a continuous order-reversing function,
 \item $+$ is an upper continuous partial function with $\dom(+) = \{(x,y) \in X^2 \mid y \leq i(x)\}$,
 \item $\star$ is a lower continuous partial function with $\dom(\star) = \{(x,y) \in X^2 \mid i(x) \nleq y\}$,
 \item $+$ and $\star$ are order preserving in both coordinates,
 \item for any $(x,y) \in \dom(\star)$,
 \begin{equation}\tag{$\star+$}\label{eq:starplus}
 x \star y = \inf \{x + w \mid (x,w) \in \dom(+) \text{ and } w \nleq y\},
 \end{equation}
\item for any $x \in X$, the image of the left translation map $L_x=L_{+, x} \colon {\downarrow} i(x) \to {\uparrow} x$ is a totally-ordered subset of ${\uparrow} x$, and moreover this function has an upper adjoint $L_x^\sharp \colon {\uparrow} x \to {\downarrow}i(x)$.
 \end{enumerate}
\end{dfn}
\begin{rem}
Note that in Definition \ref{def:mspace}, if $A$ is a bounded distributive lattice and $X$ is the Priestley dual of $A$, then the domain of $\star$ may be equivalently expressed as $\dom(\star)=\{(x,y)\in X^2 \mid \nu(y)\leq\mu(i(x))\}$, or, in terms of prime filters and ideals of $A$, as $\{(x,y)\in X^2 \mid F_y\subseteq I_{i(x)}\}$.
\end{rem}
\begin{rem}
Note that, in a \m-space, both operations $i$ and $\star$ can be defined from the partial operation $+$. Indeed, for any $x \in X$, $i(x)$ is the maximum $y \in X$ such that $(x,y) \in \dom(+)$, and the property (\ref{eq:starplus}) uniquely determines $\star$ in terms of $+$ and $i$. Thus, the data $(X, +)$ with $X$ a Priestley space and $+$ a partial function entirely determines the \m-space $(X,i,+,\star)$. However, stating the definition of \m-spaces in this smaller signature is cumbersome. More importantly, the axiom (MV6) can be expressed dually as a first-order property of the partial operations $+$ and $\star$, including in the language their domains as primitive (see Section~\ref{sec:mv} below).
\end{rem}
We will now show how to define a \m-algebra from a \m-space (Prop.~\ref{prop:spacetoalgebra}), and that any \m-algebra is isomorphic to its double dual (Prop.~\ref{prop:doubledual}). First, in the following definition, starting from a \m-space $X$, we define operations $f_+$ and $f_\star$ on the complete lattice $C$ of down-sets of the poset underlying the \m-space. We then show in Prop.~\ref{prop:spacetoalgebra} that $C$, equipped with these operations, is isomorphic to the canonical extension of a \m-algebra structure on the dual lattice of $X$. The operations $f_+$ and $f_\star$ respectively capture the $\sigma$- and $\pi$-extensions of the operation $\m$ when the canonical extension is realized as the complete lattice of down-sets.

\begin{dfn}\label{dfn:spacetoalgebra}
Let $(X,i,+,\star)$ be a \m-space and let $C$ be the complete lattice of down-sets of $X$.
Denote by $\pi_1 \colon X \times X \to X$ the projection on the first coordinate.
For $u, v \in C$, define
\begin{align*}
f_+(u,v) &:= \pi_1\left[ +^{-1}(u) \cap (X \times (X \setminus v)) \right] \\
&= \{x \in X \mid \text{there exists } w \in X \text{ such that } w \not\in v, (x,w) \in \dom(+) \text{ and } x + w \in u\} \\ \intertext{and}
f_\star(u,v) &:= X \setminus \left( \pi_1 \left[ (X \setminus \star^{-1}(u)) \cap (X \times v) \right] \right) \\
 &= \{x \in X \mid \text{for all } y \in X, \text{ if } y \in v, \text{ then } (x,y) \in \dom(\star) \text{ and } x \star y \in u\}.
\end{align*}
We call the tuple $(C,f_+,f_\star)$ the \emph{complex algebra} of $(X,i,+,\star)$.
\end{dfn}

In the proof of item (6) of Proposition~\ref{prop:spacetoalgebra}, we will need the well-known fact that the topology of open down-sets of a Priestley space is contained in the dual Scott topology of the underlying poset (see, e.g., \cite{Pr1994}). Since we do not use this terminology anywhere else in the paper, we give a direct statement and proof for the convenience of the reader.
\begin{lem}\label{lem:spectralinscott}
Let $X$ be a Priestley space, let $S \subseteq X$ be a down-directed subset, and let $U \subseteq X$ be an open down-set. If the infimum of $S$ lies in $U$, then $S$ intersects $U$ non-trivially.
\end{lem}
\begin{proof}
Without loss of generality, we may identify $X$ with the prime ideals of a distributive lattice $L$, ordered by inclusion and equipped with the Priestley topology. Now, since $S \subseteq X$ is down-directed, its intersection $s_0 = \bigcap S$ is a prime ideal, and this prime ideal must then be the infimum of $S$ in $X$. If $s_0 \in U$, then $s_0 \in \widehat{a} \subseteq U$ for some $a \in L$. By definition, this means that $a \not\in s_0$, so that $a \not\in s$ for some $s \in S$. Thus, $s \in \widehat{a} \subseteq U$, showing that $S \cap U \neq \emptyset$.
\end{proof}

\begin{prop}\label{prop:spacetoalgebra}
Let $(X,i,+,\star)$ be a \m-space with complex algebra $(C,f_+,f_\star)$. Moreover,
 denote by
 $B$ 
the sublattice of $C$ consisting of the clopen down-sets of $X$.
The following properties hold.
\begin{enumerate}
\item $f_+$ is a well-defined binary operation on $C$, which preserves $\bigvee$ in the first coordinate and sends $\bigwedge$ to $\bigvee$ in the second coordinate.
\item $f_\star$ is a well-defined binary operation on $C$, which preserves $\bigwedge$ in the first coordinate and sends $\bigvee$ to $\bigwedge$ in the second coordinate.
\item For any $u, v \in C$, $f_+(u,v) \leq f_\star(u,v)$.
\item For any $u, v_1, v_2 \in C$, $f_+(u,v_1) \wedge f_+(u,v_2) = f_+(u,v_1 \vee v_2)$.
\item For any $u \in C$, $f_+(u,0_B) = u$.
\item For any $a, b \in B$, if $b \neq 0_B$, then $f_+(a,b) = f_\star(a,b)$.
\item The pair $(B,\m_B)$, with $a \m_B b := f_+(a,b)$, is a well-defined \m-algebra.
\end{enumerate}
\end{prop}
\begin{proof}
(1) Let $u, v \in C$. We need to show that $f_+(u,v)$ is a down-set. Let $x \in f_+(u,v)$ and $x' \leq x$. Pick $w \not\in v$ with $(x,w) \in \dom(+)$ and $x + w \in u$. Since $\dom(+)$ is a down-set, $(x',w) \in \dom(+)$, and since $+$ is order-preserving in the first coordinate, $x' + w \leq x + w$. Since $u$ is a down-set, $x' + w \in u$. Hence, $x' \in f_+(u,v)$. The preservation properties are clear from discrete duality for operators, or easily verified directly.

(2) Similar to (1).

(3) Let $u,v \in C$. Let $x \in f_+(u,v)$ be arbitrary, and pick $w \not\in v$ with $(x,w) \in \dom(+)$ and $x + w \in u$. Let $y \in v$ be arbitrary. Then $w \nleq y$, since $v$ is a down-set, and $y \in v$ but $w \not\in v$. In particular, since $(x,w) \in \dom(+)$ means that $w \leq i(x)$, we must have $i(x) \nleq y$, i.e., $(x,y) \in \dom(\star)$. Moreover, $x \star y \leq x + w$, by property (\ref{eq:starplus}) in the definition of \m-spaces. Hence, we have $x \star y \in u$. Since $y \in v$ was arbitrary, we conclude that $x \in f_\star(u,v)$.

(4) Let $u, v_1, v_2 \in C$. The inclusion from right to left is clear because $f_+$ is order reversing in the second coordinate. Let $x \in f_+(u,v_1) \wedge f_+(u,v_2)$. For $i \in \{1,2\}$, pick $w_i \not\in v_i$ with $(x,w_i) \in \dom(+)$ and $x + w_i \in u$. Since the image of $L_x$ is totally ordered, without loss of generality assume $x + w_1 \leq x + w_2$. Let $w_0 := L_x^\sharp(x+w_2)$. By the definition of $L_x^\sharp$, we have $w_1, w_2 \leq w_0$, so $w_0 \not\in v_1 \vee v_2$. Also by the definition of $L_x$, $x + w_0 = x + w_2$, so $x + w_0 \in u$, so we obtain $x \in f_+(u,v_1 \vee v_2)$.

(5) Let $u \in C$. Suppose that $x \in u$. By definition of \m-spaces, pick $y_x \in X$ such that $(x,y_x) \in \dom(+)$ and $x + y_x = x$. It follows that $x \in f_+(u,0_B)$. Conversely, suppose that $x \in f_+(u,0_B)$. By definition of $f_+$, pick $w \in X$ such that $(x,w) \in \dom(+)$ and $x + w \in u$. Since $x \leq x + w$ in any \m-space, and $u$ is a down-set, we have $x \in u$.

(6) Let $a, b \in B$ with $b \neq 0_B$. By item (3), it suffices to prove that $f_\star(a,b) \leq f_+(a,b)$. Let $x \in f_\star(a,b)$. Let $y\in b$ be arbitrary. By definition of $f_\star(a,b)$, we have $(x,y) \in \dom(\star)$ and $x \star y \in a$. 
Note that $\{x + w \ | \ (x,w) \in \dom(+), w \nleq y\}$ is a totally ordered set with infimum $x \star y \in a$.
By Lemma~\ref{lem:spectralinscott}, we may choose $w_y \nleq y$ with $(x,w_y) \in \dom(+)$ and $x + w_y \in a$. Also, since $w_y \nleq y$, choose $b_y \in B$ such that $y \in b_y$ and $w_y \not\in b_y$. We then have $b \subseteq \bigcup_{y \in b} b_y$, so pick a finite subcover $\{b_{y_i}\}_{i=1}^n$. Since $b \neq 0_B$, we have $n \geq 1$. For each $i \in \{1,\dots,n\}$, we have $x + w_{y_i} \in a$, and $w_{y_i} \not\in b_{y_i}$, so $x \in f_+(a,b_{y_i})$ for each $i \in \{1,\dots,n\}$. Now, using item (4), we have $x \in \bigwedge_{i=1}^n f_+(a, b_{y_i}) =  f_+(a,\bigvee_{i=1}^n b_{y_i})$, and using item (1) $f_+(a,\bigvee_{i=1}^n b_{y_i}) \leq f_+(a,b)$. Hence, $x \in f_+(a,b)$, as required.

(7) We first need to prove that $f_+(a,b)$ is clopen for every $a,b \in B$. Let $a$ and $b$ be clopen down-sets in $X$. If $b = 0$, then $f_+(a,b) = a$ by (5), which is clopen. Now assume $b \neq 0$. Note that, since $\pi_1$ is a continuous map between compact Hausdorff spaces, it is a closed map. By the definition of $f_+$ and continuity of $+$, it follows  that $f_+(a,b)$ is closed. Similarly, by the definition of $f_\star$ and continuity of $\star$, $f_\star(a,b)$ is open. By (6), $f_+(a,b) = f_\star(a,b)$, so this set is clopen. It follows immediately from (1), (2), and (6) that $\m_B$ is a double quasioperator of type $(1,\mathsf{op})$. The fact that $\m_B$ is normal as an operator follows from (1) by taking the empty join in the first coordinate and the empty meet in the second coordinate. Finally, the fact that $0_B$ is a right-unit for $\m_B$ follows from (5). Hence, $(B,\m_B)$ is a \m-algebra.
\end{proof}

\begin{dfn}
For any \m-space $(X,i,+,\star)$, we call $(B,\m_B)$, defined as in the last item of Proposition~\ref{prop:spacetoalgebra}, the \emph{dual \m-algebra} of $(X,i,+,\star)$.
\end{dfn}

\begin{prop}\label{prop:doubledual}
Let $(A,\m_A)$ be a \m-algebra, let $(X,i,+,\star)$ be its extended dual space, and let $(B,\m_B)$ be the dual \m-algebra of $(X,i,+,\star)$. The \m-algebras $(A,\m_A)$ and $(B,\m_B)$ are isomorphic via the Stone-Priestley isomorphism $\hat{(\cdot)} \colon A \to B$. Moreover, this isomorphism extends uniquely to an isomorphism between $(A^\delta, \m_A^\sigma)$ and the reduct $(C,f_+)$ of the complex algebra of $(X,i,+,\star)$.
\end{prop}
\begin{proof}
For the first statement, we need to prove that $\hat{a \m_A b} = \hat{a} \m_B \hat{b}$, for any $a, b \in A$. Indeed, for any $x \in X$, we have:
\begin{align*}
x \in \hat{a \m_A b} &\iff a \m_A^\sigma b \nleq \mu(x) \\
			          &\iff a \m_A^\sigma \bigwedge_{b\leq\mu(w)}\mu(w) \nleq \mu(x) \\
			          &\iff \bigvee_{b\leq\mu(w)} a \m_A ^\sigma\mu(w) \nleq \mu(x) \\
			          &\iff \exists w \in X \text{ such that } b\leq\mu(w) \text{ and }a\m_A \mu(w)\nleq \mu(x) \\
			          &\iff \exists w \in X \text{ such that } b\leq\mu(w) \text{ and }a\nleq \mu(x)\m_A^{\sigma\sharp} \mu(w)\\
			          &\iff \exists  w \in X \text{ such that } b\leq\mu(w), \mu(x)\m_A^{\sigma\sharp}\mu(w)\neq 1, \text{ and }a\nleq \mu(x)\m_A^{\sigma\sharp} \mu(w)\\
			          &\iff \exists w \in X \text{ with } w \not\in \hat{b}, (x,w) \in \dom(+), \text{ and } x + w \in \hat{a} \\
			          &\iff x \in f_+(\hat{a},\hat{b}) = \hat{a} \m_B \hat{b}.
\end{align*}
For the `moreover' part, we first recall \cite[Thm.~3.7]{GeJo2004} that the (smooth) extension of any surjective homomorphism between distributive lattice expansions is a homomorphism between the canonical extensions. Since extensions also preserve injectivity \cite[Thm. 3.2]{GeJo2004}, the (smooth) extension of the isomorphism $\hat{(\cdot)}$ is an isomorphism between $(A^\delta, \m_A^\sigma)$ and $(C, \m_B^\sigma)$.
Since  $\varphi(a\m_B b) = f_+|_{B \times B}(a,b)$, it thus remains to prove that $f_+$ is the $\sigma$-extension of $\m_B$.
Note that completely join-irreducibles (respectively, completely meet-irreducibles) in $C$ are of the form $\downset x$ (respectively, $(\upset x)^\mathsf{c}$), where $x\in X$. Thus for each $x\in X$ we have $\varphi(\mu(x))=(\upset x)^\mathsf{c}$ and $\varphi(\nu(x))=\downset x$. Since both $\m_B^\sigma$ and $f_+$ are complete $(1,\mathsf{op})$-operators, it therefore suffices to prove that, for any $x,y \in X$, $\varphi(\nu(x) \m_B^\sigma \mu(y)) = f_+(\downset x,(\upset y)^\mathsf{c})$. Let $x, y \in X$. Recall that, by definition of $\varphi$,
\begin{align*}
\varphi(\nu(x) \m_B^\sigma \mu(y)) &= \bigwedge \{a \m_B b \ | \ \varphi(\nu(x)) \subseteq a, b \subseteq \varphi(\mu(y)), a, b \in B\} \\
&= \bigwedge \{f_+(a,b) \ | \ \varphi(\nu(x)) \subseteq a, b \subseteq \varphi(\mu(y)), a, b \in B\}.
\end{align*}
It is immediate from this equality that $f_+(\downset x, (\upset y)^\mathsf{c}) \subseteq \varphi(\nu(x) \m_B^\sigma \mu(y))$. For the other inequality, suppose that $z \not\in f_+(\downset x, (\upset y)^\mathsf{c})$. Denote by $L_{+,z}=L_z \colon X \rightharpoonup X$ the partial function $w\mapsto z + w$ as usual. It follows from the continuity of $+$ that $L_z$ is continuous. Also, by definition of $f_+$, $z \not\in f_+(\downset x, (\upset y)^\mathsf{c})$ means that $L_z^{-1}({\downarrow}x)$ is disjoint from ${\uparrow} y$. Since $L_z^{-1}({\downarrow}x)$ is the filtered intersection of the collection of clopen down-sets $\{L_z^{-1}(a) \ | \ x \in a, a \in B\}$ and $\upset y$ is the filtered intersection of the collection of clopen up-sets $\{ X \setminus b \ | \ y \not\in b, b \in B\}$, by compactness of $X$ there exist $a, b \in B$ such that $x \in a$, $y \not\in b$ and $L_z^{-1}(a)$ is disjoint from $X \setminus b$. By definition of $f_+$, this means that $z \not\in f_+(a,b)$, so $z \not\in \varphi(\nu(x) \m_B^\sigma \mu(y))$, as required.
\end{proof}

\subsection{Morphisms}\label{sec:morphisms}
The above shows that any \m-algebra can be represented as the dual algebra of an \m-space. This representation theorem can be extended to a duality theorem, as we will do now.
Before we state and prove the duality theorem, we need to define the appropriate notion of morphism between \m-spaces. The correct notion is that of `bounded morphism for $+$'. Indeed, readers familiar with duality for modal and/or Heyting algebras will recognize the `forth' and `back' conditions in, respectively, items (2) and (3) of the following definition.
\begin{dfn}\label{dfn:morphisms}
A \emph{morphism} from a \m-space $(X_1,i_1,+_1,\star_1)$ to a \m-space $(X_2,i_2,+_2,\star_2)$ is a continuous order-preserving function $f \colon X_1 \to X_2$ such that
\begin{enumerate}
\item for all $x \in X_1$, $f(i_1(x)) = i_2(f(x))$,
\item for all $x,y \in X_1$, if $(x,y) \in \dom(+_1)$, then $f(x) +_2 f(y) \leq f(x +_1 y)$,
\item for all $x \in X_1$ and $z \in X_2$, if $(f(x),z) \in \dom(+_2)$, then there exists $w' \in X_1$ such that $(x,w') \in \dom(+_1)$, $z \leq f(w')$, and $f(x +_1 w') = f(x) +_2 z$.
\end{enumerate}
\end{dfn}

\begin{rem}
In the special case of Definition~\ref{dfn:morphisms} where the Priestley space $X_1$ is a closed subspace of the Priestley space $X_2$, and the order on $X_1$ is the restriction of the order on $X_2$, we have that the inclusion map $\iota \colon X_1 \hookrightarrow X_2$ is a morphism of \m-spaces if, and only if, the following two properties hold:
\begin{enumerate}
\item for all $x \in X_1$, $i_1(x) = i_2(x)$;
\item for all $(x,y) \in \dom(+_1)$, $x +_1 y = x +_2 y$.
\end{enumerate}
\end{rem}
\begin{lem}\label{lem:morphisms}
Let $(X_1,i_1,+_1,\star_1)$ and $(X_2,i_2,+_2,\star_2)$ be \m-spaces with dual \m-algebras $(A_1, \m_1)$ and $(A_2,\m_2)$, respectively. Let $f \colon X_1 \to X_2$ be a continuous order-preserving function with dual lattice homomorphism $h = f^{-1} \colon A_2 \to A_1$. The following are equivalent:
\begin{enumerate}
\item $f$ is a morphism of \m-spaces;
\item $h$ is a homomorphism of \m-algebras.
\end{enumerate}
\end{lem}
\begin{proof}
Let $h^\delta \colon A_2^\delta \to A_1^\delta$ be the extension of $h$ to a complete homomorphism between canonical extensions; writing $\varphi_i$ for the isomorphism $A_i^\delta \cong \mathcal{D}(X_i)$, we have that $h^\delta = \varphi_1^{-1} \circ f^{-1} \circ \varphi_2$. Because $\m$ is a monotone operation, the equation $h(a \m b) = h(a) \m h(b)$ is canonical \cite[Lemma 3.24]{GeJo2004}, so that item (2) is equivalent to
\begin{enumerate}
\item[3.] for all $u, v \in A_1^\delta$,  $h^\delta(u) \m_1^\sigma h^\delta(v) = h^\delta(u \m_2^\sigma v)$,
\end{enumerate}
Note that, if $u = 0$ or $v = 1$, then $h^\delta(u) \m_1^\sigma h^\delta(v) = 0 = h(0) = h^\delta(u \m_2^\sigma v) = h(0)$ always holds. Since the functions $(u,v) \mapsto h(u \m^\sigma v)$ and $(u,v) \mapsto h(u) \m^\sigma h(v)$ preserve non-empty joins in the first coordinate and send non-empty meets to joins in the second coordinate, (3) is equivalent to
\begin{enumerate}
\item[4.] for all $y,z \in X$, $h^\delta(\nu(y)) \m_1^\sigma h^\delta(\mu(z)) = h^\delta(\nu(y) \m_2^\sigma \mu(z))$.
\end{enumerate}
Using the isomorphism from Proposition~\ref{prop:doubledual} and the isomorphisms $\varphi_i$, (4) is equivalent to
\begin{enumerate}
\item[5.] for all $y, z \in X$, $f_{+_1}(f^{-1}({\downarrow}y), f^{-1}(({\uparrow}z)^c)) = f^{-1}(f_{+_2}({\downarrow}y,({\uparrow}z)^c))$.
\end{enumerate}
Using the definitions of $f^{-1}$ and $f_{+_i}$, (5) is equivalent to
\begin{enumerate}
\item[6.] for all $x,y,z \in X$, there exists $w' \in f^{-1}({\uparrow}z)$ such that $(x,w') \in \dom(+_1)$ and $f(x +_1 w') \leq y$, if, and only if, there exists $w \in {\uparrow}z$ such that $(f(x),w) \in \dom(+_2)$ and $f(x) +_2 w \leq y$.
\end{enumerate}

By first-order logic, (6) is equivalent to the conjunction of
\begin{enumerate}
\item[7a.] for all $x,y,z,w' \in X$, if $z \leq f(w')$, $(x,w') \in \dom(+_1)$ and $f(x +_1 w') \leq y$, then there exists $w \in X$ such that $z \leq w$, $(f(x),w) \in \dom(+_2)$ and $f(x) +_2 w \leq y$, and
\item[7b.] for all $x,y,z,w \in X$, if $z \leq w$, $(f(x),w) \in \dom(+_2)$ and $f(x) +_2 w \leq y$, then there exists $w' \in X$ such that $z \leq f(w')$, $(x,w') \in \dom(+_1)$ and $f(x +_1 w') \leq y$,
\end{enumerate}
We now claim that (7a) and (7b) are, respectively, equivalent to
\begin{enumerate}
\item[8a.] for all $x,w' \in X$, if $(x,w') \in \dom(+_1)$, then $(f(x),f(w')) \in \dom(+_2)$ and $f(x) +_2 f(w') \leq f(x +_1 w')$, and
\item[8b.] for all $x,z \in X$, if $(f(x),z) \in \dom(+_2)$, then there exists $w' \in X$ such that $z \leq f(w')$, $(x,w') \in \dom(+_1)$ and $f(x +_1 w') \leq f(x) +_2 z$.
\end{enumerate}
Indeed, (8a) is the special case of (7a) where we put $y := f(x +_1 w')$ and $z := f(w')$: since $\dom(+_2)$ is a down-set and $+_2$ is order-preserving, the existence of $w$  in (7a) implies in particular that $(f(x),f(w')) \in \dom(+_2)$ and $f(x) +_2 f(w') \leq f(x) +_2 w \leq y = f(x +_1 w')$. Also, (8b) is the special case of (7b) where we put $w := z$ and $y := f(x) +_2 z$. Conversely, if (8a) holds, and $x,y,z,w'$ are as in the premise of (7a), set $w := f(w')$, then $(f(x),f(w')) \in \dom(+_2)$ by (8a), and $f(x) +_2 w = f(x) +_2 f(w') \leq f(x +_1 w') \leq y$, so this $w$ satisfies the conclusion of (7a). If (8b) holds, and $x,y,z,w$ are as in the premise of (7b), then $(f(x),z) \in \dom(+_2)$ since $\dom(+_2)$ is a down-set, so we can pick $w' \in X$ as in (8b). Then $f(x +_1 w') \leq f(x) +_2 w \leq y$, so the same $w'$ satisfies the conclusion of (7b).

Applying (8a) to $w' := i_1(x)$ gives that $(f(x),f(i_1(x))) \in \dom(+_2)$, i.e., $f(i_1(x)) \leq i_2(f(x))$. Applying (8b) to $z := i_2(f(x))$ gives $w' \in X$ with $z \leq f(w')$ and $w' \leq i_1(x)$, so $z \leq f(i_1(x))$, i.e., $i_2(f(x)) \leq f(i_1(x))$. Thus, (8a) and (8b) together imply condition (1) in Definition~\ref{dfn:morphisms}. Note moreover that in the presence of $f(i_1(x)) = i_2(f(x))$, (8a) is equivalent to condition (2) in Definition~\ref{dfn:morphisms}: if $(x,w') \in \dom(+_1)$, then $f(w') \leq f(i_1(x)) = i_2(f(x))$, since $f$ is order preserving, so $(f(x),f(w')) \in \dom(+_2)$. Also note that, in the presence of (8a), the element $w'$ which exists according to (8b) actually satisfies $f(x +_1 w') \leq f(x) +_2 z \leq f(x) +_2 f(w') \leq f(x +_1 w')$, so equality holds throughout. We thus conclude that (2), which is equivalent to the conjunction of (8a) and (8b), is equivalent to $f$ being a morphism of $\m$-spaces as defined in Definition~\ref{dfn:morphisms}.
\end{proof}

Combining Proposition~\ref{prop:spacetoalgebra}, Proposition~\ref{prop:doubledual} and Lemma~\ref{lem:morphisms} show that the assignment which sends a \m-space to its dual \m-algebra and a morphism of \m-spaces to the inverse image homomorphism between the dual \m-algebras is a well-defined, full, faithful and essentially surjective functor. We have therefore proved the following theorem.

\begin{thm}\label{thm:mduality}
The category of \m-algebras and homomorphisms is dually equivalent to the category of \m-spaces and morphisms.
\end{thm}

\section{MV-algebras} \label{sec:mv}
The aim of this section is to specialize the duality for $\m$-algebras obtained in the previous section to the full subcategory of \m-algebras that is isomorphic to the category of MV-algebras. This will in particular yield a duality between MV-algebras and certain \m-spaces satisfying some additional first-order conditions. Background references for MV-algebras are \cite{cdm,Mun2011}.

\subsection{MV-algebras as \m-algebras}
An \emph{MV-algebra} is an algebra $(A,\oplus,\neg,0)$ of type $(2,1,0)$ satisfying the following conditions, where $1$ abbreviates $\neg 0$:
\begin{enumerate}
\item $(A,\oplus,0)$ is a commutative monoid,
\item $\neg\neg a = a$ for all $a\in A$,
\item $a\oplus 1 = 1$ for all $a\in A$, and
\item $\neg(\neg a \oplus b)\oplus b = \neg(\neg b \oplus a)\oplus a$ for all $a,b\in A$.
\end{enumerate}
In fact, although none of our results here rely on this fact, we note that it already follows from the remaining axioms that $\oplus$ is commutative \cite{Kol2013}. Because it is often given as the sixth item in a commonly-adopted equational basis for MV-algebras, (4) is frequently called (MV6) in the literature (see, e.g., the influential monograph \cite{cdm}).

Many of the characteristic properties of MV-algebras derive from (MV6), including the fact that the term $a\vee b := \neg(\neg a \oplus b)\oplus b$ defines the join operation of a lattice for any MV-algebra. This lattice has least element $0$ and greatest element $1$, and its meet operation is definable via the De Morgan dual $a\wedge b := \neg (\neg a\vee\neg b)$. With respect to order of this lattice, $\oplus$ has a right co-residual $\m$ that is definable via the term $a\m b := \neg (\neg a \oplus b)$.

It follows that to any MV-algebra $\mathbf{A} = (A, \oplus, \neg, 0)$, one may associate an algebra $\mathbf{A}^\m = (A,\vee,\wedge,0,1,\m)$, where the operations of $\mathbf{A}^\m$ are obtained in the manner just described. Note that $\mathbf{A}^\m$ is a $\m$-algebra in the sense of Definition~\ref{dfn:malgebras}. Conversely, if $\mathbf{A} = (A,\vee,\wedge,0,1,\m)$ is a $\m$-algebra, then one may define an algebra $\mathbf{A}^\oplus = (A, \oplus,\neg, 0)$ in the signature of MV-algebras by putting $\neg a := 1 \m a$ and $a \oplus b := \neg(\neg a \m b)$. The algebra $\mathbf{A}^\oplus$ need not be an MV-algebra in general; we now characterize those $\m$-algebras for which it is.

\begin{prop}\label{prop:newmvminus}
Let $\mathbf{A} = (A, \vee,\wedge,0,1,\m)$ be a $\m$-algebra. The following are equivalent:
\begin{enumerate}
\item The algebra $\mathbf{A}^\oplus$ is an MV-algebra;
\item For all $a,b,c \in A$,
\begin{enumerate}
\item[(i)] $(a\m b)\m c = a\m \neg (\neg b\m c)$,
\item[(ii)] $\neg a\m b = \neg b\m a$, and
\item[(iii)] $a \wedge b = a \m (a \m b)$.
\end{enumerate}
\end{enumerate}
Consequently, if the equivalent conditions (1) and (2) are satisfied, then $\oplus$ is an associative and commutative binary operation, $\m$ is the right co-residual of $\oplus$, and $\neg$ is an involution.
\end{prop}

\begin{proof}
Suppose first that (1) holds, and let $a,b,c\in A$. Note that $\neg\neg a=a$ holds by definition.  For (i), the associativity of $\oplus$ gives $(\neg a \oplus b)\oplus c = \neg a\oplus (b\oplus c)$, and rewriting this using the identities $\neg\neg x = x$ and $x\m y=\neg (\neg x\oplus y)$ yields $(a\m b)\m c = a\m \neg (\neg b\m c)$. For (ii), the commutativity of $\oplus$ gives $\neg (\neg a\m b) = \neg (\neg b\m a)$, whence $\neg a\m b =\neg b\m a$. Finally, for (iii) note that $\neg\neg x = x$, (MV6), the commutativity of $\oplus$,  and the identity $a\vee b = \neg (\neg a\oplus b)\oplus b$ provide that
\begin{align*}
a\wedge b &= \neg[(b\oplus\neg a)\oplus\neg a]\\
&= \neg [\neg a\oplus \neg (\neg a\oplus b)]\\
&= a\m (a\m b)
\end{align*}
as desired.

For the converse, suppose that $\mathbf{A}$ is an $\m$-algebra satisfying the three conditions of (2). Property (i) gives that $(\neg a\m b)\m c = \neg a\m\neg (\neg b\m c)$ for all $a,b,c\in A$. Instantiating $a=1$ in (iii) gives $\neg\neg x=x$ for all $x\in A$, whence we obtain $\neg [\neg\neg (\neg a\m b)\m c] = \neg [\neg a\m \neg (\neg b\m c)]$. Applying the definition of $\oplus$ then yields $(a\oplus b)\oplus c=a\oplus (b\oplus c)$, so $\oplus$ is associative. The commutativity of $\oplus$ follows from (ii). That $0$ is a neutral element for $\oplus$ follows from the fact that $a\m 0 = a$ holds in any $\m$-algebra together with the identity $\neg\neg a = a$. This shows that $(A,\oplus,0)$ is a commutative monoid. The identity $a\oplus 1 = 1$ follows from the fact that $a\m 1=0$ in any $\m$-algebra. To see that $\neg(\neg a \oplus b)\oplus b = \neg(\neg b \oplus a)\oplus a$ holds, note that (iii) together with the commutativity of $\oplus$ and $\neg\neg a = a$ shows that both sides of the equation are equal to $a\vee b$, whence the result follows.
\end{proof}

It follows from the remarks in the paragraph preceding Proposition~\ref{prop:newmvminus}, and from the fact that a function preserves $\ominus$ and $1$ iff it preserves $\oplus$ and $\neg$, that the category of MV-algebras is \emph{isomorphic} to the full subcategory of $\m$-algebras which validate the equivalent conditions in Proposition~\ref{prop:newmvminus}. In the presence of the identities (2)(i), (2)(ii), and $\neg\neg a = a$, the identity (2)(iii) is equivalent to (MV6) for the defined operation $\oplus$. Even when $\m$ is the right co-residual of the defined operation $\oplus$ and $\neg$ is an involution, (2)(i) and (2)(ii) do not suffice to axiomatize MV-algebras without (2)(iii). The following proposition characterizes the duals of this larger class of $\m$-algebras.

\begin{prop}\label{prop:resminusduals}
Let $\mathbf{A} = (A, \vee,\wedge,0,1,\m)$ be a $\m$-algebra and $(X,i,+,\star)$ and its extended dual space. Moreover, for $a,b\in A$ define $a\oplus b = \neg (\neg a\m b)$. The following are equivalent.
\begin{enumerate}
\item For all $a,b,c\in A$,
\begin{enumerate}
\item[(i)] $\neg\neg a = a$.
\item[(ii)] $a\oplus b = b\oplus a$.
\item[(iii)] $(a\oplus b)\oplus c = a\oplus (b\oplus c)$.
\item[(iv)] $\m$ is the right co-residual of $\oplus$.
\end{enumerate}
\item For all $x,y,z\in X$,
\begin{enumerate}
\item[(i)] $i(i(x)) = x$.
\item[(ii)] If $(x,y), (y,x)\in\dom(+)$, then $x + y = y + x$.
\item[(iii)] If $(x,y+x), (y,z), (x+y,z), (x,y)\in\dom(+)$, then $x + (y + z) = (x + y) + z$.
\item[(iv)] If $(i(x),y), (z,y)\in\dom(+)$, then $i(x) + y \leq i(z)$ if and only if $z + y\leq x$.
\end{enumerate}
\end{enumerate}
\end{prop}

\begin{proof}
It is easy to see that for each $x\in X$ and $a\in A$, $a\in I_{i(i(x))}$ if and only if $\neg\neg a\in I_x$. Hence if (1)(i) holds, then (2)(i) follows immediately. Conversely, if (1)(i) fails then there exists $a\in A$ with $a\neq \neg\neg a$, and there is some prime ideal $I_x$ that contains one of $a$ or $\neg\neg a$ but not the other. This implies that (2)(i) fails, giving that (1)(i) and (2)(i) are equivalent.

Now suppose that (1) holds. Then $\m^\sigma$ is the right co-residual of $\oplus^\pi$ (see, e.g., \cite[Prop.~5.3]{GGM2014} for a proof), and the latter is associative and commutative since $\oplus$ is associative and commutative, using Proposition~\ref{prop:canonicity}. Thus, $\m^{\sigma\sharp} = \oplus^\pi$ is associative and commutative, whence (2)(ii) and (2)(iii) follow immediately from the definition of $+$. For (2)(iv), note that:
\begin{align*}
\nu(x)\m^\sigma\mu(y) &= \neg^\delta (\neg^\delta \nu(x)\oplus^\pi\mu(y))\\
&= \neg^\delta (\mu (i(x)) \m^{\sigma\sharp} \mu (y))\\
&= \neg^\delta \mu(i(x) + y)\\
&= \nu (i(i(x) + y))
\end{align*}
Now observe:
\begin{alignat*}{3}
&&\nu(x)\m^\sigma \mu(y)\leq \mu(z) &\iff \nu(x)\leq \mu(z)\m^{\sigma\sharp}\mu(y)\\
\text{iff  }\;\;\;&& \nu(i(i(x)+y))\leq\mu(z) &\iff \nu(x)\leq\mu(z+y)\\
\text{iff  }\;\;\;&& \nu(z)\leq\nu(i(i(x)+y)) &\iff \nu(z+y)\leq \nu(x)\\
\text{iff  }\;\;\;&& z\leq i(i(x)+y) &\iff z+y\leq x\\
\text{iff  }\;\;\;&& i(x)+y\leq i(z) &\iff z+y\leq x.
\end{alignat*}
This gives (2)(iv), showing that (1) implies (2).

For the converse, note that (2)(ii) and (2)(iii) entail that $\m^{\sigma\sharp}$ is associative and commutative. In the presence of (1)(i) (equivalently, (2)(i)) we have that $a\m b = \neg (\neg a\oplus b)$. Thus, for all $x,y\in X$,
$$\nu(x)\m^\sigma\mu(y) = \neg^\delta (\neg^\delta \nu(x)\oplus^\pi\mu(y)).$$
Since $\m^{\sigma\sharp}$ is associative and commutative, then $\neg^\delta$ being an involution implies that for all $x,y\in X$,
$$\nu(x)\m^\sigma\mu(y) = \neg^\delta (\neg^\delta \nu(x)\m^{\sigma\sharp}\mu(y)),$$
whence $\m^{\sigma\sharp}$ and $\oplus^\pi$ coincide. It follows that the restriction of $\oplus^\pi$ to $A$ is associative and commutative, i.e., $\oplus$ is associative and commutative.  The argument of the preceding paragraph then shows that $\m^\sigma$ is the right co-residual of $\oplus^\pi$ by (2)(iv), whence $\m$ is the right co-residual of $\oplus$. It follows that (2)(iv) implies (1)(iv), and hence (2) implies (1).
\end{proof}

\subsection{The dual of (MV6)}\label{subsec:mv6}
From Proposition~\ref{prop:newmvminus}, the algebras whose duals are characterized in Proposition~\ref{prop:resminusduals} comprise a supervariety of MV-algebras, and MV-algebras are exactly the subvariety of algebras satisfying $a \wedge b = a \m (a \m b)$ (which is equivalent to (MV6) in this formulation). The following dualizes the latter identity in this context. Crucially, offering a dual condition in terms of both $+$ and $\star$ allows for the application of Proposition \ref{prop:canonicity} in the proof of the following. The fact that this is possible underscores the benefit of working with the duality we have offered here.

\begin{prop}\label{prop:mv6}
Let $(A,\m)$ be a $\m$-algebra with extended dual space $(X,i,+,\star)$. Assume further that $\neg$ is an involution, that the operation $a \oplus b := \neg(\neg a \m b)$ is a right co-residual of $\m$, and that $\oplus$ is commutative.
The following are equivalent:
\begin{enumerate}
\item for all $a, b \in A$, $a \wedge b = a \m (a \m b)$;
\item for all $x,x',y \in X$, if $(x,y) \in \dom(\star)$, and there exists $w \nleq y$ such that $x' + w \leq x \star y$, then $x' \leq x$.
\end{enumerate}
\end{prop}
\begin{proof}
Note that (1) may readily be seen to be equivalent to the condition that for all $a, b,\in A$, $\neg a \wedge b \leq (a \oplus b) \m a$. By Proposition \ref{prop:canonicity}, this inequality is equivalent to
\begin{enumerate}
\item[3.] for all $u,v \in A^\delta$, $\neg u \wedge v \leq (u \oplus^\pi v) \m^\pi u$.
\end{enumerate}
Note that, if $u = 0$, then $\neg u \wedge v = v$, and $(u \oplus^\pi v) \m^\pi u = v$, where in the second equality one uses that the equations $a \oplus 0 = a$ and $a \m 0 = a$, which hold in any \m-algebra, is canonical. Thus, (3) is equivalent to:
\begin{enumerate}
\item[4.] for all $u,v \in A^\delta$, if $u \neq 0$ then $\neg u \wedge v \leq (u \oplus^\pi v) \m^\pi u$.
\end{enumerate}
Now, since the completely join-irreducibles join-generate $A^\delta$ and $\m^\pi$ sends non-empty joins to non-empty meets in the second coordinate, (4) is equivalent to:
\begin{enumerate}
\item[5.] for all $u,v \in A^\delta$, $x, y \in X$, if
$\left\{ \begin{array}{c} u \neq 0,\\ \nu(x) \leq \neg u \wedge v, \\ \nu(y) \leq u \end{array}\right.$
then $\nu(x) \leq (u \oplus^\pi v) \m^\pi \nu(y)$.
\end{enumerate}
Note that the minimum values of $u$ and $v$ for which the antecedent of (5) is satisfied, if any, are $u := \nu(y)$ and $v := \nu(x)$. Therefore, since $u$ and $v$ only occur positively in the term $(u \oplus^\pi v) \m^\pi \nu(y)$, (5) is equivalent to:
\begin{enumerate}
\item[6.] for all $x, y \in X$, if $\nu(x) \leq \neg \nu(y)$ then $\nu(x) \leq (\nu(x) \oplus^\pi \nu(y)) \m^\pi \nu(y)$.
\end{enumerate}
(The equivalence of (5) and (6) is the typical ``Sahlqvist'' correspondence argument, well-known in modal logic.) Condition (6) is essentially already a first-order condition on the extended dual space $X$. It remains to show that (6) is equivalent to the simpler condition (2). To this end, first note that (6) is equivalent to:
\begin{enumerate}
\item[7.] for all $x, y \in X$, if
$\left\{ \begin{array}{c} \nu(x) \leq \neg \nu(y), \\ \nu(x) \leq \mu(x'), \\ \nu(y) \leq \mu(w), \end{array}\right.$
then $\nu(x) \m^{\pi\flat} \nu(y) \leq \mu(x') \oplus^\pi \mu(w)$.
\end{enumerate}
Indeed, to see that (6) and (7) are equivalent, consecutively use the adjunction between $\m^{\pi\flat}$ and $\m^\pi$, approximate $\nu(x)$ as the meet of completely meet-irreducibles $\mu(x')$ and $\nu(y)$ as the meet of completely meet-irreducibles $\mu(w)$, and then use the fact that $\oplus^\pi$ preserves arbitrary meets in both coordinates, being the right co-residual of $\m^\sigma$.

Substituting the definitions of the operations $+$ and $\star$ and the defining properties of $\mu$ and $\nu$ in (7), and rearranging using first-order logic, (7) is equivalent to:
\begin{enumerate}
\item[8.] for all $x,x',y \in X$, if $(x,y) \in \dom(\star)$, and there exists $w \nleq y$ such that $x' + w \leq x \star y$, then $x' \leq x$.
\end{enumerate}
The above condition is precisely (2) as required.
\end{proof}

\begin{dfn}
We say that a $\m$-space $(X, i, +, \star)$ is an \emph{MV-space} if, for all $x,y,z\in X$,
\begin{enumerate}
\item[(i)] $i(i(x)) = x$.
\item[(ii)] If $(x,y), (y,x)\in\dom(+)$, then $x + y = y + x$.
\item[(iii)] If $(x,y+x), (y,z), (x+y,z), (x,y)\in\dom(+)$, then $x + (y + z) = (x + y) + z$.
\item[(iv)] If $(i(x),y), (z,y)\in\dom(+)$, then $i(x) + y \leq i(z)$ if and only if $z + y\leq x$.
\item[(v)] If $(x,y) \in \dom(\star)$, and there exists $w \nleq y$ such that $z + w \leq x \star y$, then $z \leq x$.
\end{enumerate}
\end{dfn}
\begin{cor}\label{cor:MValgdual}
The category of MV-algebras with MV-algebra homomorphisms is dually equivalent to the full subcategory of $\m$-spaces consisting of the MV-spaces.
\end{cor}

\begin{proof}
The result is immediate from Theorem \ref{thm:mduality} and Propositions \ref{prop:resminusduals} and \ref{prop:mv6} because the latter precisely characterize the $\m$-spaces dual to a class of $\m$-algebras that are term-equivalent to MV-algebras, by Proposition~\ref{prop:newmvminus}.
\end{proof}

Note that in particular, Corollary \ref{cor:MValgdual} establishes that the duals of MV-algebras may be captured relative to the theory of $\m$-spaces by simple first-order conditions.

\section{Some examples}\label{sec:examples}

In this final section, we present two examples illustrating our duality for $\m$-algebras. The first of these concerns a class of $\m$-algebras satisfying conditions (i)--(iv) of Corollary \ref{cor:MValgdual}, but refuting (v), which dualizes (MV6). In contrast, the second example discusses a perfect MV-algebra, and hence satisfies (i)--(v) of Corollary \ref{cor:MValgdual}. Both examples arise from the disconnected rotation construction (see \cite{J2004}, and for a duality-theoretic discussion see \cite{FU2019}), and share the same underlying lattice reduct. This lattice is constructed as follows (see Figure \ref{fig:NM}). Consider the non-positive integers $N = \{n\in\mathbb{Z} \mid n\leq 0\}$ equipped with the obvious ordering. Define
$$A=(\{1\}\times N) \cup (\{0\}\times N),$$
and order the elements of $A$ by $(j,a)\leq (k,b)$ if and only if one of the following hold:
\begin{enumerate}
\item $j<k$,
\item $j=k=1$ and $a\leq b$, or
\item $j=k=0$ and $b\leq a$.
\end{enumerate}
The set $A$ is a chain under $\leq$, and $(A,\leq)$ gives the (lattice reduct of) the \emph{disconnected rotation} of $N$. Because $(A,\leq)$ is a chain, each proper principal downset $\downset (j,a)$ of $(A,\leq)$ is a prime ideal. Apart from these, there is just one more prime ideal, usually called the \emph{co-radical} of $A$:
$$\mathcal{C}=\{(0,a)\mid a\in N\}.$$
The poset of prime ideals of the lattice $(A,\leq)$ is pictured in Figure \ref{fig:NM}. Note that here we write principal downsets of the form $\downset (0,a-1)$ as $(\upset (1,a))^\mathsf{c}$ owing to the fact that, in both examples to follow, we have $i(1,a) = (\upset (1,a))^\mathsf{c}$.
\begin{figure}
\begin{center}
\begin{tikzpicture}

\tikzset{vertex/.style = {shape=circle,draw,fill=black, inner sep = 1.5pt}}
\tikzset{edge/.style = {-,> = latex'}}
\tikzset{dote/.style = {-,dotted,> = latex'}}

\node[vertex,label=left:\tiny{${0}$}] (a) at  (0,0) {};
\node[vertex,label=left:\tiny{${-1}$}] (b) at  (0,-1) {};
\node[vertex,label=left:\tiny{${-2}$}] (c) at  (0,-2) {};

\draw[edge] (a) to (b);
\draw[edge] (b) to (c);

\draw[dote] (c) to (0, -3);

\end{tikzpicture}
\hspace{0.5 in}
\begin{tikzpicture}

\tikzset{vertex/.style = {shape=circle,draw,fill=black, inner sep = 1.5pt}}
\tikzset{edge/.style = {-,> = latex'}}
\tikzset{dote/.style = {-,dotted,> = latex'}}

\node[vertex,label=left:\tiny{${(1,0)}$}] (a) at  (0,0) {};
\node[vertex,label=left:\tiny{${(1,-1)}$}] (b) at  (0,-1) {};
\node[vertex,label=left:\tiny{${(1,-2)}$}] (c) at  (0,-2) {};

\node[vertex,label=left:\tiny{${(0,-2)}$}] (d) at  (0,-4) {};
\node[vertex,label=left:\tiny{${(0,-1)}$}] (e) at  (0,-5) {};
\node[vertex,label=left:\tiny{${(0,0)}$}] (f) at  (0,-6) {};

\draw[edge] (a) to (b);
\draw[edge] (b) to (c);
\draw[edge] (d) to (e);
\draw[edge] (e) to (f);

\draw[dote] (c) to (0, -3);
\draw[dote] (d) to (0, -3);

\end{tikzpicture}
\hspace{0.5 in}
\begin{tikzpicture}

\tikzset{vertex/.style = {shape=circle,draw,fill=black, inner sep = 1.5pt}}
\tikzset{edge/.style = {-,> = latex'}}
\tikzset{dote/.style = {-,dotted,> = latex'}}

\node[vertex,label=left:\tiny${\downset (1,-1)}$] (a) at  (0,0) {};
\node[vertex,label=left:\tiny${\downset (1,-2)}$] (b) at  (0,-1) {};
\node[vertex,label=left:\tiny${\downset (1,-3)}$] (c) at  (0,-2) {};

\draw[edge] (a) to (b);
\draw[edge] (b) to (c);
\draw[dote] (c) to (0,-3);

\node[vertex,label=left:$\tiny{\mathcal{C}}$] (o) at (0,-3.5) {};
\node[vertex,label=left:\tiny${(\upset (0,-3))^\mathsf{c}}$] (p) at  (0,-5) {};
\node[vertex,label=left:\tiny${(\upset (0,-2))^\mathsf{c}}$] (q) at  (0,-6) {};
\node[vertex,label=left:\tiny${(\upset (0,-1))^\mathsf{c}}$] (r) at  (0,-7) {};

\draw[dote] (p) to (0,-4);

\draw[edge] (p) to (q);
\draw[edge] (q) to (r);
\end{tikzpicture}
\end{center}

\caption{Hasse diagrams of $N$ (left), its disconnected rotation $(A,\leq)$ (middle), and the Priestley dual of $(A,\leq)$ (right).}
\label{fig:NM}
\end{figure}
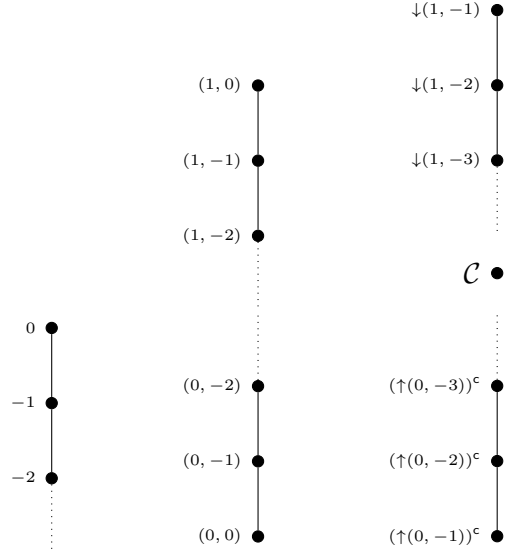

\begin{exa}[Nilpotent minimum algebras]\label{ex:NM}
A \emph{monoidal t-norm logic algebra} (or \emph{MTL-algebra}) is an algebra $(A,\wedge,\vee,\cdot,\to,0,1)$, where $(A,\wedge,\vee,0,1)$ is a bounded distributive lattice, $(A,\cdot,1)$ is a commutative monoid, $\to$ is the residual of $\cdot$, and for all $a,b\in A$,
$$(a\to b)\vee (b\to a) = 1.$$
MTL-algebras form a variety, and the latter identity axiomatizes the fact that this variety is generated by totally-ordered MTL-algebras. An MTL-algebra is \emph{involutive} if $a\mapsto a\to 0$ is an involution, and in this setting one may define a co-residuated commutative and associative operation $\oplus$ by $a\oplus b = \neg (\neg a\cdot\neg b)$. If $(A,\wedge,\vee,\cdot,\to,0,1)$ is an involutive MTL-algebra and $\m$ is the right co-residual of $\oplus$, then $(A,\wedge,\vee,\m,0,1)$ is an $\m$-algebra. A \emph{nilpotent minimum algebra} (introduced by Fodor \cite{Fod1995}) is an involutive MTL-algebra satisfying the identity
$$(a\cdot b \to 0)\vee (a\wedge b\to a\cdot b) = 1.$$ The directly indecomposable nilpotent minimum algebras lacking a negation fixed point can be obtained as disconnected rotations of G\"odel hoops \cite{U2018}.

Endowing the chain $N$ (defined above) the operation
\[ a\Rightarrow b = \begin{cases}
      0 & \text{ if }a\leq b \\
      b & \text{ if }b<a
   \end{cases}
\]
gives this lattice with the structure of a Brouwerian algebra (i.e., a Heyting algebra possibly missing a bottom element). The disconnected rotation ${\bf A}$ of the Brouwerian algebra $N$ is constructed as follows. The lattice reduct of the ${\bf A}$ is the lattice $(A,\leq)$ defined above, and we endow it with binary operations $\cdot$ and $\to$ defined by
\[ (j,a)\cdot (k,b) = \begin{cases}
      (1, a\wedge b) & \text{ if } j=k=1 \\
      (0,0) & \text{ if }j=k=0 \\
      (0,a\Rightarrow b) & \text{ if }k<j\\
      (0,b\Rightarrow a) & \text{ if }j<k
   \end{cases}
\]
\[ (j,a)\to (k,b) = \begin{cases}
      (1, a\Rightarrow b) & \text{ if } j=k=1 \\
      (0,b\Rightarrow a) & \text{ if }j=k=0 \\
      (0,a\wedge b) & \text{ if }k<j \\
      (1,0) & \text{ if }j<k.
   \end{cases}
\]
The resulting algebra ${\bf A}=(A,\wedge,\vee,\cdot,\to,(0,0),(1,0))$ is a nilpotent minimum algebra. With the defined operation $a\m b = a\cdot\neg b = a\cdot (b\to (0,0))$, we obtain a (term-equivalent) $\ominus$-algebra $(A,\wedge,\vee,\m,(0,0),(1,0))$. Direct computation yields that $\neg (0,a) = (1,a)$ and $\neg (1,a) = (0,a)$.

To describe the $\m$-space dual to ${\bf A}$, it is convenient to work with prime ideals. For prime ideals $I,J$ and points $x,y$ of the dual space of ${\bf A}$ with $I=I_x$ and $J=I_y$, we will abbreviate $I_{i(x)}$, $I_{x+y}$ and $I_{x\star y}$ by $i(I)$, $I+J$ and $I\star J$, respectively. One may show by direct computation that $i(\downset (1,a))=(\upset (0,a))^\mathsf{c}$, and that $i(\mathcal{C})=\mathcal{C}$. This suffices to characterize $i$ since this operation is an involution.

Computing with Lemma \ref{lem:plusstardef}(2) shows that:
\[ I+J=J+I = \begin{cases}
      I\vee J & \text{ if }J\subseteq i(I) \\
      \text{undefined} & \text{ otherwise. }
   \end{cases}
\]
For instance, $\downset (1,a) + I$ is defined if and only if $I\subseteq (\upset (0,a))^\mathsf{c}$. Consequently, $I=(\upset (0,b))^\mathsf{c}$ for some $b\geq a$. Then:
\begin{align*}
\downset (1,a) + I &= \{(j,p)\in A \mid \exists (k,q)\in (\upset (0,b))^\mathsf{c}[(j,p)\m (k,q)\in\downset (1,a)]\}\\
&= \{(j,p)\in A \mid \exists (k,q)\not\geq (0,b)[(j,p)\cdot\neg (k,q)\leq (1,a)]\}
\end{align*}
Note that $(k,q)\not\geq (0,b)$ if and only if $k=0$ and $q\not\leq b$. Also, $(j,p)\cdot\neg (0,q)\leq (1,a)$ trivially holds if $j=0$, so we need only consider the case when $j=1$. Then $(j,p)\cdot\neg (0,q)\leq (1,a)$ if and only if $p\wedge q\leq a$, and by residuated this holds if and only if $p\leq q\Rightarrow a$. Subject to $q\not\leq b$, we have $a\leq b < q$ and hence $q\Rightarrow a = a$. It follows that $\downset (1,a) + I = \downset (1,a) = \downset (1,a)\vee I$ for any ideal $I$ for which $+$ is defined. The other cases follow by similar computations.

The values of $\star$ may be computed directly from Lemma \ref{lem:plusstardef} as well, but we use Proposition \ref{prop:starminplus}. Phrased in terms of prime ideals, for all $(I,J)\in\dom(\star)$:
\begin{align*}
I\star J&=\inf \{I+K\mid (I,K)\in\dom(+) \text{ and }K\not\subseteq J\}\\
&=\inf \{I + K \mid  J\subset K \subseteq i(I)\}.
\end{align*}
To express the values of $\star$ compactly, for $I\neq\downset (1,-1),\mathcal{C}$ we denote by $I'$ the unique cover of $I$. Using the above characterization of $\star$, one may compute that:
\[ I\star J = \begin{cases}
      I & \text{ if } J\subseteq i(I)\subseteq I\text{ or }J\subset I\subset\mathcal{C}\\
      J' & \text{ if } I\subset\mathcal{C}\text{ and }I\subset J\\
      \text{undefined} & \text{ otherwise. }
   \end{cases}
\]
For instance, suppose that $I=(\upset (0,a))^\mathsf{c}$ for some $0\neq a\in N$ and $J=\downset (1,b)$ for some $b < a$. Then $J'$ is the least $K$ satisfying $J\subset K\subseteq i(I)$, whence $I\star J = J'$. On the other hand, we have
\begin{align*}
I\star \mathcal{C}&=\inf \{I + K \mid  \mathcal{C}\subset K \subseteq i(I)\}\\
&= \inf \{I\vee K \mid \mathcal{C}\subset K\}\\
&= \mathcal{C}.
\end{align*}
The remaining cases are handled in an analogous manner.

We note that the dual of ${\bf A}$ satisfies the conditions given in Proposition \ref{prop:resminusduals}. However, condition (v) of Corollary \ref{cor:MValgdual} fails. To see this, take, for example, $I_x = (\upset (0,-2))^\mathsf{c}$, $I_z = (\upset (0,-3))^\mathsf{c}$, $I_y=\downset (1,-4)$, and $w=i(z)$. Then $z+w = w = x\star y$ and $w\not\leq y$, but $z\not\leq x$.
\end{exa}

\begin{exa}[The Chang MV-algebra]\label{ex:Chang}
Let ${\bf A} = (A,\oplus,\neg,0)$ be an MV-algebra and let $\cdot$ be binary operation defined by $a\cdot b = \neg (\neg a\oplus\neg b)$. We say that an element $a\in A$ \emph{has finite order} if there exists a positive integer $n$ such that $a^n = 0$, and that $a$ \emph{has infinite order} otherwise. We say that ${\bf A}$ is \emph{perfect} if for each $a\in A$, $a$ has finite order if and only if $\neg a$ has infinite order (see, e.g., \cite{BDiNL1993}). The variety generated by the perfect MV-algebras coincides with the variety generated by the \emph{Chang MV-algebra} \cite{DiNL1994}, and the perfect MV-algebras are exactly those that are isomorphic to disconnected rotations of cancellative hoops \cite{NEG2005}. The Chang MV-algebra ${\bf C}$ may be defined on the same lattice reduct $A$ as in the previous example. The operations $\oplus$ and $\neg$ are uniquely determined by:
$$\neg (1,a) = (0,a),$$
$$(1,a)\oplus (1,b)=(1,0),$$
$$(0,a)\oplus (0,b)=(0,a+b),$$
$$(1,a)\oplus (0,b)=(1,\min\{a-b,0\}).$$
The elements of the form $(0,b)$ are ``infinitesimals'' in the sense that there is no $n>0$ such that the sum $(0,b)\oplus\dots\oplus (0,b)$ ($n$-times) is the top element.

Computing as before shows $i$ is the same as in the previous example. Moreover, the partial operation $+$ on the dual of ${\bf C}$ is determined by
$$\downset (0,a) + \downset (0,b) = \downset (0,b)+\downset (0,a) = \downset (0,a+b)$$
$$\mathcal{C}+\downset (0,a) = \downset (0,a) + \mathcal{C} = \mathcal{C},\text{ and }$$
$$\downset (0,a) +\downset (1,b) =\downset (1,b) + \downset (0,a) = \downset (1,a-b)\text{ for } a<b,$$
and is undefined in all other cases. The values of $\star$ may again be computed using Proposition \ref{prop:starminplus}, and this yields:
$$\downset (0,a)\star \downset (0,b) = \downset (0,b)\star\downset (0,a) = \downset (0,a+b-1),$$
$$\downset (0,a)\star\mathcal{C} = \mathcal{C}\star\downset (0,a) =\mathcal{C},$$
$$\downset (0,a)\star\downset (1,b) = \downset (1,a-b-1) \text{ if }b<a,\text{ and }$$
$$\downset (1,a)\star \downset (0,b) = \downset (1,a-b+1)\text{ if }b<a,$$
and is undefined in the remaining cases.

We note that by McNaughton's representation theorem \cite{McN1951}, the free MV-algebra $F_{MV}(1)$ on one generator may be realized as the MV-subalgebra of $[0,1]^{[0,1]}$ whose members are piecewise linear with integer coefficients. The Chang MV-algebra ${\bf C}$ is isomorphic to the quotient of $F_{MV}(1)$ by the prime MV-ideal consisting of all $f\in F_{MV}(1)$ such that $f\restriction_{[0,\epsilon)}=0$ for some $\epsilon>0$. The results in \cite{GGM2014} show that the MV-space associated to $F_{MV}(1)$ admits a decomposition over its space of prime MV-ideals, and the latter is well-known in the literature (see, e.g., \cite{Mun1986}). From this perspective, what we have computed above is the MV-space dual to one of the ``vertical'' stalks in one of the sheaf representations discussed in \cite{GGM2014}. The explicit computation of the dual space of $F_{MV}(1)$ (and, more generally, $F_{MV}(n)$) can thus be performed by analogously computing the dual spaces of quotients of $F_{MV}(1)$ at prime MV-ideals.
\end{exa}

\bibliographystyle{amsplain}

\end{document}